\documentclass[p10]{siamltex}
\usepackage{amsfonts,amscd,amsmath, amssymb,t1enc,latexsym,mathrsfs}
\usepackage{pstricks,pst-plot,pst-node,pst-grad,pst-3d, hhline, multirow, graphicx}
\usepackage{color}
\DeclareGraphicsExtensions{.jpg}

\newtheorem{remark}{Remark}

\newcommand{\p}{\partial}

\newcommand{\Grad}{\nabla}
\newcommand{\Div}{\nabla\cdot}

\newcommand{\wtilde}{\widetilde}
\newcommand{\what}{\widehat}

\newcommand{\bq}{\begin{equation}}
\newcommand{\eq}{\end{equation}}
\newcommand{\bt}{\begin{theorem}}
\newcommand{\et}{\end{theorem}}
\newcommand{\bqa}{\begin{eqnarray*}}
\newcommand{\eqa}{\end{eqnarray*}}

\def\pd#1#2{\frac{\partial #1}{\partial#2}}
\def\cK{\mathcal K}

\title{Immersed Finite Element Method for Eigenvalue Problem}

\author{ {Seungwoo Lee}\footnotemark[1] \and {Do Y. Kwak}\footnotemark[1]
        \and Imbo Sim \footnotemark[2]}
\begin{document}

\maketitle
\renewcommand{\thefootnote}{\fnsymbol{footnote}}
\footnotetext[1]{Department of Mathematical Science, Korea Advanced Institute of Science and Technology, 305-701 Daejeon, Republic of Korea.}
\footnotetext[2]{National Institute for Mathematical Sciences, 305-811 Daejeon, Republic of Korea $\quad$ ({\tt imbosim@nims.re.kr}).}

\begin{abstract}
We consider the approximation of elliptic eigenvalue problem with an immersed interface.  The main aim of this paper is to prove the stability and convergence of an immersed finite element method (IFEM) for eigenvalues using Crouzeix-Raviart $P_1$-nonconforming approximation. We show that spectral analysis for the classical eigenvalue problem can be easily applied to our model problem. We analyze the IFEM for elliptic eigenvalue problem with an immersed interface and derive the optimal convergence of eigenvalues. Numerical experiments demonstrate our theoretical results.
\end{abstract}



\pagestyle{myheadings}
\thispagestyle{plain}


\section{Introduction}
In this paper, we consider the approximation of elliptic eigenvalue problem with an immersed interface. The interface problems are often encountered in fluid dynamics, electromagnetics, and materials science. Especially, elastic waves propagating in heterogeneous media with interfaces occur in materials science \cite{Deak-Ahmed, Zhang-Leveque}. Also, electromagnetic problems with different conductivity or permeability often arise in optical waveguides \cite{Badia-Codina, Hiptmair-Li-Zou}. The main difficulty in solving such problems is caused mainly by the non-smoothness of solution across the interface. One choice to overcome it is to use finite element methods based on fitted meshes along the interface. Another choice is to use any meshes independent of interface geometry for the computational domain. In the latter case, LeVeque and Li \cite{LeVeque-Li} introduce the immersed interface method based on the finite difference method where the jump conditions are properly incorporated in the scheme. However, the resulting linear system of equation from this method may not be symmetric and positive definite \cite{Li-Lin-Wu}. On the other hand, the immersed finite element method (IFEM) has been developed where the local basis functions are constructed to satisfy the jump conditions \cite{Li-Lin-Wu} and its variants have been analyzed \cite{Chou-Kwak-Wee,Hou-Liu,Kwak-W-C,Li-Lin-Lin-Rogers}. The related work in this direction can be found in \cite{Chang-Kwak,Gong-Li-Li, Lin-Sheen-Zhang} and references therein.

The purpose of this paper is to prove the stability and convergence of an immersed finite element method for eigenvalues using Crouzeix-Raviart $P_1$-nonconforming approximation \cite{Kwak-W-C}.
As a model problem, we consider the eigenvalue problem with an immersed interface, i.e.
\begin{align}
-\nabla \cdot (\beta \nabla u) &= \lambda u \quad \; \text{in} \quad \;\Omega^{+} \cup \Omega^{-},  \nonumber\\
\, [u]_\Gamma &= 0, \quad \left[\beta\frac{\partial u}{\partial n} \right]_\Gamma = 0,  \label{eq:modelEq} \\
 u &= 0 \qquad \text{on} \quad \partial\Omega \nonumber,
 \end{align}
 where  $\Omega$ is a convex polygonal domain in $\mathbb{R}^2$ which is separated
 into two subdomains $\Omega^+ $ and $\Omega^-$ by a $C^2$-interface $\Gamma = \partial \Omega^- \subset
 \Omega$ with $\Omega^+ = \Omega \setminus \Omega^-$. The symbol $[\,\cdot\,]_\Gamma$ denotes the jump across $\Gamma$. The coefficient $\beta(x)$ is a discontinuous function bounded below and above by two positive constants. For the sake of simplicity, we assume that the coefficient $\beta$ is a positive piecewise constant, that is,
$$\beta(x) = \left\{
\begin{aligned}
\beta^- \quad \text{for} \; x \in \Omega^-, \\
\beta^+ \quad \text{for} \; x \in \Omega^+.
\end{aligned}
\right.$$

The $P_1$-nonconforming FEM is widely used in solving elliptic equations and is shown to be useful in solving the mixed formulation of elliptic problems \cite{Arnold-Brezzi} and the Stokes equations \cite{Crouzeix-Raviart}.  Recently, Kwak et al. \cite{Kwak-W-C} introduced an IFEM based on the piecewise $P_1$-nonconforming polynomials and they proved optimal orders of convergence in the $H^1$ and $L^2$-norm.

There have been various mathematical studies of finite element methods for eigenvalue problems. A unified approach to a posteriori and a priori error analysis for finite element approximations of self-adjoint elliptic eigenvalue problems is presented in \cite{Larson}. The convergence of an adaptive method for elliptic eigenvalue problems is proved in \cite{Giani-Graham}. For a nonconforming approximation, Dari et al. \cite{Dari-Duran-Padra} prove a posteriori error analysis of the eigenvalue. The study of mixed eigenvalue problems can be found in \cite{Boffi2007, Boffi-Brezi-Gastaldi, Mercier-Osborn-Rappaz-Raviart}. To our best knowledge, spectral and convergence analysis of IFEM for eigenvalue problems with immersed interface has not been done so far. It is worth emphasizing that the spectral properties of eigenvalue problems with immersed interface play key roles in the analysis and simulation for more complicated problems, such as fluid-structure interactions, moving interfaces and the numerical stability for PDEs.

In this work, we analyze the IFEM for elliptic eigenvalue problems with immersed interface and derive the optimal convergence of eigenvalues. Furthermore, we show that spectral analysis for the classical eigenvalue problem can be easily applied to our model problem. In particular, the spectral approximation of Galerkin methods can be proved by using fundamental properties of compact operators in Banach space.  Such an investigation originates from a series of papers of Osborn and Babu\v ska \cite{Babuska-Osborn, Osborn}. It has been extended in \cite{Descloux-Nassif-Rappaz1978-1,Descloux-Nassif-Rappaz1978-2}  to estimate Galerkin approximations for noncompact operators. Further application to discontinuous Galerkin approximations has been developed by Buffa et al \cite{Antonietti-Buffa-Perugia}. We formulate the eigenvalue problem with immersed interface in terms of compact operators in order to understand the spectral behavior. The analysis presented in this paper is carried out along the lines of the references \cite{Descloux-Nassif-Rappaz1978-1, Descloux-Nassif-Rappaz1978-2}.

The paper is structured as follows. In the next section, we give a brief review on $P_1$-nonconforming IFEM \cite{Kwak-W-C}. In Section 3, we introduce a modified version of IFEM with an additional term and formulate the eigenvalue problem with the immersed interface. Section 4 contains the analysis of the spectral approximation which is proved to be spurious-free. The approximation is proved by  means of basic results from the theory of compact operator in Banach space.  In section 5 we derive the convergence rate of eigenvalues based on $P_1$-nonconforming IFEM. In the final section, we demonstrate numerical experiments for the model problem which corroborate the theoretical results in the preceding sections.

\newpage

\section{Preliminaries}
We consider an elliptic interface problem corresponding to the model problem (\ref{eq:modelEq}):
\begin{eqnarray} \label{eq}
-\Div(\beta\Grad u) &=& f ~~ \mathrm{ in}~ \;\Omega^{+} \cup \Omega^{-}, \\
\,[u]_\Gamma&=&0, ~~~ \left[\,\beta\pd un\,\right]_\Gamma=0,  \label{flux}\\
u &=& 0  ~~ \mathrm{ on}~  \p\Omega.  \label{BC}
\end{eqnarray}
The weak formulation of the problem (\ref{eq}) - (\ref{BC}) is to find $u\in H^1_0(\Omega)$ such that
\begin{equation} \label{op}
\int_\Omega \beta \nabla u\cdot \nabla v dx = \int_\Omega f v dx , ~~~ \forall v
\in H^1_0(\Omega)
\end{equation}
with $f \in L^2(\Omega)$.
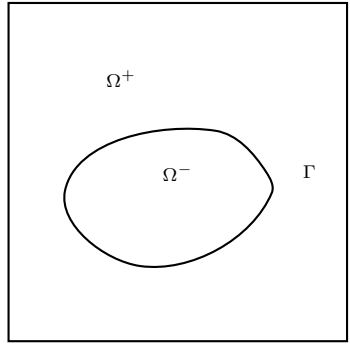
\begin{figure}[ht]
\begin{center}
      \psset{unit=2.5cm}
      \begin{pspicture}(-1,-1)(1,1)
        \pspolygon(0.9,0.9)(-0.9,0.9)(-0.9,-0.9)(0.9,-0.9)
        \psccurve(0.47,0) (0.2,0.22)(-0.6,-0.1)(-0.2,-0.5)(0.5,-0.11)
        \rput(0,0){\scriptsize$\Omega^-$}
        \rput(-0.3,0.5){\scriptsize$\Omega^+$}
        \rput(0.7,0){\scriptsize$\Gamma$}
      \end{pspicture}
\caption{Domain $\Omega$ with interface $\Gamma$} \label{fig:domain1}
\end{center}
\end{figure}

We begin by introducing a Sobolev space which is convenient for describing the regularity of the solution of the elliptic interface problem (\ref{eq}) - (\ref{BC}).
For a bounded domain $D$, we let  $H^m(D) = W^m_2(D)$ be the usual Sobolev space of order $m$ with semi-norm and norm denoted by $|\cdot|_{m,D}$ and  $\|\cdot\|_{m,D}$, respectively.
We define the space
\begin{eqnarray*}
\wtilde{H}^m(\Omega) := \{ u\in H^{m-1}(\Omega)\, :\, u\in H^m(\Omega^s),
s=+,- \}
\end{eqnarray*}
equipped with the norm
\begin{eqnarray*}
\|u\|^2_{\wtilde{H}^m(\Omega)} :=
\|u\|^2_{H^{m-1}(\Omega)}+\|u\|^2_{H^m(\Omega^+)} +
\|u\|^2_{H^m(\Omega^-)},~~ \forall\, u\in\wtilde{H}^m(\Omega).
\end{eqnarray*}
By Sobolev embedding theorem, for any $u \in H^2(\Omega)$, we have $u \in W^1_s(\Omega), \;\forall s > 2$. Then we have following regularity theorem for the weak solution $u$ of
the variational problem (\ref{op}); see \cite{Bramble-King} and \cite{Ladyzenskaja-Rivkind-Uralceva}.
\begin{theorem} \label{thm:reg}
The variational problem (\ref{op}) has a unique
solution $u\in\wtilde{H}^2(\Omega)$ which satisfies for some constant $C>0$
\begin{eqnarray}
\|u\|_{\wtilde{H}^2(\Omega)} \leq C \|f\|_{0,\Omega}. \nonumber
\end{eqnarray}
\end{theorem}

We now describe an immersed finite element method (IFEM) based on Crouzeix-Raviart element \cite{Kwak-W-C}.
 Let $\{\mathcal{K}_h\}$ be the usual quasi-uniform triangulations of the domain $\Omega$ by the triangles of maximum diameter $h$. Note that we do not require an element $K \in \mathcal K_h$ to be aligned with the interface $\Gamma$.
We assume the following situations:
\begin{itemize}
\item the interface intersects the edges of an element at no more than two points 
\item the interface intersects each edge at most once, except possibly it passes through two vertices.
\end{itemize}  For a smooth interface, those assumptions are satisfied if $h$ is sufficiently small.
We call an element $K\in\mathcal{K}_h$ an \textit{interface element} if the interface $\Gamma$
passes through the interior of $K$, otherwise $K$ is a
\textit{non-interface element}.
 We denote by $\mathcal K_h^*$ the collection of all interface elements. We may replace $\Gamma\cap K$ by the line segment joining two intersection points on the edges of each $K\in \cK_h$.

  For each $K\in \mathcal{K}_h$  and non-negative integer $m$, let
\begin{eqnarray*}
\wtilde{H}^m(K) &:=& \{\,u\in L^2(K) : \,u|_{K\cap \Omega^s}\in
H^m(K\cap \Omega^s), s = +,-\,\},
\end{eqnarray*}
equipped with norms
\begin{eqnarray*}
|u|^2_{m,K} &:=& |u|^2_{m,K\cap \Omega^+} + |u|^2_{m,K\cap \Omega^-},\\
\|u\|^2_{m,K} &:=& \|u\|^2_{m,K\cap \Omega^+} + \|u\|^2_{m,K\cap \Omega^-}.
\end{eqnarray*}
  To deal with the interface conditions in the model problem (\ref{eq:modelEq}), we introduce the following spaces,
 \begin{eqnarray*}
\wtilde{H}^2_{\Gamma}(K) &:=&  \left\{\,u\in H^1(K) : \, u|_{K\cap
\Omega^s}\in H^2(K\cap \Omega^s),\,s = +,-~\, \text{and} \left[\beta \pd un\right]_\Gamma = 0 \text{ on } \Gamma\cap K\, \right\},\\
\wtilde{H}^2_{\Gamma}(\Omega) &:=& \left\{\,u\in H^1_0(\Omega) :
 \, u|_{K}\in\wtilde{H}^2_{\Gamma}(K),\,\, \forall K\in \mathcal{K}_h \right\}.
\end{eqnarray*}
Clearly, $\wtilde{H}^2_{\Gamma}(K)$ and $\wtilde{H}^2_{\Gamma}(\Omega)$ are subspace of $\wtilde{H}^2(K)$ and $\wtilde{H}^2(\Omega)$, respectively.

As usual, we  construct local basis functions on each
element $K$ of the triangulation $\mathcal{K}_h$. We let
$$\overline v|_e= \frac1{|e|}\int_{e}{v}\,ds$$ denote the average of a function $v\in H^1(K)$ along an edge $e$.
For a non-interface
element $K\in\mathcal{K}_h$, we simply use the standard linear
shape functions whose degrees of freedom are determined by average values on the edges. Let $N_h(K)$ denote the linear space spanned by the three basis functions $\phi_i$ satisfying $\overline{\phi_i}|_{e_j} = \delta_{ij}$ for $i,j=1,2,3$.
The $P_1$-nonconforming space $N_h(\Omega)$ is given by
$$ N_h(\Omega)= \left\{
\begin{aligned}
 \phi: \phi|&_K\in P_1(K) \mbox{ for } K\in\cK_h\setminus \cK_h^*;\ \mbox{if $K_1,K_2 \in \mathcal K _h$ share an edge $e$,} \\
 &\text{then} \int_{e}{\phi}|_{\partial K_1} ds= \int_{e}{\phi}|_{\partial K_2}  ds; \mbox{ and }
  \int_{\partial K \cap \partial\Omega}{\phi}\,ds=0
\end{aligned}
\right\}.$$

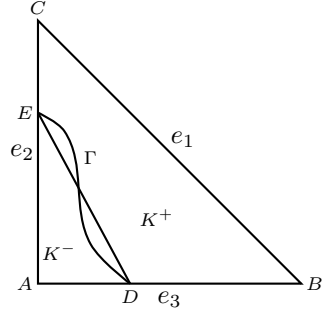
\begin{figure}[ht]
  \begin{center}
    \psset{unit=3.5cm}
    \begin{pspicture}(0,0)(1,1)
      \psset{linecolor=black} \pspolygon(0,0)(1,0)(0,1) \psline(0,0.65)(0.35,0)
      \pscurve(0,0.65)(0.1,0.58)(0.2,0.15)(0.35,0)
      \rput(0,1.05){\scriptsize$C$}
      \rput(-0.05,0){\scriptsize$A$}
      \rput(1.05,0){\scriptsize$B$}
      \rput(0.5,-0.06){$e_3$}
      \rput(0.55,0.55){$e_1$}
      \rput(-0.06,0.5){$e_2$}
      \rput(-0.05,0.65){\scriptsize$E$}
      \rput(0.08,0.12){\scriptsize$K^-$}
      \rput(0.45,0.25){\scriptsize$K^+$}
      \rput(0.35,-0.05){\scriptsize$D$}
      \rput(0.20,0.48){\scriptsize$\Gamma$}
\pnode(-.3,0.6){a}
\pnode(0.12,0.5){b}

\pnode(-.3,.3){a}
\pnode(0.22,0.21){b}
    \end{pspicture}
    \caption{Reference interface triangle} \label{fig:interel}
\end{center}
\end{figure}

Now we consider a reference interface element $K$ and assume that the interface $\Gamma$ intersects the edges of an element $K$ at $D$ and $E$ as in Figure \ref{fig:interel}.
Given a linear function $\phi = V_1\phi_1 + V_2\phi_2 + V_3\phi_3$ on $K$ where  $V_i \in \mathbb{R}$, $\phi_i,\,(i=1,2,3)$ are the
standard basis functions \cite{Crouzeix-Raviart}.
 We construct a new basis function $\hat{\phi}$ which holds the same degrees of freedom as $\phi$. Additionally, the function $\hat{\phi}$ should be linear on $K^+$ and $K^-$, and satisfy the jump conditions in (\ref{flux}).
Since the edge $e_1$ does not intersect the interface, the function $\hat{\phi}$ on the interface element $K$ can be
conveniently described as  follows:
\begin{eqnarray} \label{def:basis}
\hat{\phi} = \left\{
\begin{array}{cc}
    c_1^-\phi_1+ c_2^-\phi_2 + c_3^-\phi_3 & \text{in  $K^-$,}\\
    V_1\phi_1+ c_2^+\phi_2 + c_3^+\phi_3 & \text{in  $K^+$},
\end{array}\right.
\end{eqnarray}
satisfying
\begin{eqnarray}
&&\hat{\phi}^-(D) = \hat{\phi}^+(D),~\hat{\phi}^-(E) =
\hat{\phi}^+(E), \label{conti}\\
&& \frac1{|e_i|}\int_{e_i}{\hat\phi}\,ds=V_i,\ i=2,3,\\
&& \beta^-\pd{\hat{\phi}^-}{n\;\;\,}|_{\overline{DE}} =
\beta^+\pd{\hat{\phi}^+}{n\;\;\,}|_{\overline{DE}} . \label{fconti}
\end{eqnarray}
It turns out that the function $\hat{\phi}$ implies $ \frac1{|e_1|}\int_{e_1}{\hat\phi}\,ds=V_1$ from the second equation in (\ref{def:basis}). The modified
function $\hat{\phi}$ is uniquely determined by (\ref{def:basis}) - (\ref{fconti}) (see \cite{Kwak-W-C}).

We denote by $\what N_h(K)$ the local finite element space on the
interface element $K$ whose basis functions $\hat{\phi}_i,\,
i=1,2,3$ are defined by above construction.
 We define
the {\em immersed finite element space} $\what N_h(\Omega)$ as the collection of functions $\hat\phi \in L^2(\Omega)$ such that
\begin{itemize}
\item$ \hat\phi|_K\in \what N_h(K) \mbox{ if } K\in\cK_h^*$
\item$ \hat{\phi}|_K \in N_h(K) \mbox{ if } K\in \cK_h\setminus \cK_h^*$
\item$  \int_{e}{\hat\phi}|_{\partial K_1} ds= \int_{e}{\hat\phi}|_{\partial K_2}  ds \;\mbox{ if $K_1,K_2 \in \mathcal K _h$ share an edge $e$}$
\item$  \int_{\partial K \cap \partial\Omega}{\hat\phi}\,ds=0.$
\end{itemize}

Let $H_h(\Omega) := H^1_0(\Omega) + \what N_h(\Omega)$  be endowed with the broken $H^1$-norm $\|v\|^2_{1,h} := \sum_{K\in \mathcal{K}_h}\|v\|^2_{1,K}$. Next, we need an interpolation operator.
For any $v\in {H}^1(K)$, $ I_h v\in \what N_h(K)$ is determined by  the average values of $v$ on each edge:
$$\overline{(I_h v)}|_{e_i} = \bar{v}|_{e_i},~~ i=1,2,3. $$ We call $I_h v$ the local \emph{interpolant} of $v$ in $\what N_h(K)$. We naturally extend it to $ {H}^1(\Omega)$ by  $(I_hv)|_{K} =I_h (v|_K)$ for each $K \in \mathcal{K}_h$. Then  we have the following approximation property of the interpolation in $\wtilde{H}_{\Gamma}^2(\Omega)$ \cite{Kwak-W-C}.
\begin{theorem}\label{thm:apperror}
There exists a constant $C>0$ such that
\begin{eqnarray*}
\|v-I_h v\|_{0,\Omega} + h\|v- I_h v\|_{1,h} \leq C h^2
\|v\|_{\wtilde{H}^2(\Omega)}, \quad \forall v \in \wtilde{H}_{\Gamma}^2(\Omega).
\end{eqnarray*}
\end{theorem}

\section{Variational formulation} In this section, we consider a variational formulation for the model problem (\ref{eq:modelEq}).
Let $\Omega$,  $\Gamma$ and $\beta$ be the same as in the previous section.
Multiplying $v \in H^1_0(\Omega)$ and integrating by parts in $\Omega^\pm$, we obtain
\begin{align}
\sum_{s=\pm} \int_{\Omega^s} -\nabla\cdot(\beta\nabla u) \cdot v\, dx &= \sum_{s=\pm} \int_{\Omega^s} \beta\nabla u\cdot\nabla v \, dx - \int_{\Gamma}\left[\beta\frac{\partial u}{\partial n}\right] v \, dx \label{eq:weakpro1}\nonumber\\
 &= \int_{\Omega} \beta \nabla u\cdot \nabla v \, dx. \nonumber
\end{align}
Hence the weak formulation of the problem (\ref{eq:modelEq}) is to find the eigenvalues $\lambda \in \mathbb{C}$ and the eigenfunctions $u \in H^1_0(\Omega)$ such that
\begin{equation}
a(u,v) = \lambda(u,v), \quad \forall v \in H^1_0(\Omega),
\label{eq:dweakform}
\end{equation}
where
$$a(u,v) = \int_{\Omega} \beta\nabla u \cdot \nabla v\, dx, \quad \forall u,v \in H^1_0(\Omega).$$
Using the solution operator $T:L^2(\Omega) \rightarrow H^1_0(\Omega)$, the eigenvalue problem (\ref{eq:dweakform}) can be treated as the variational form
\begin{equation} \label{eq:T_operator}
a(Tf, v) = (f,v), \quad \forall v \in H^1_0(\Omega) \nonumber
\end{equation}
with $f\in L^2(\Omega)$.
Note that if $(\lambda, u) \in \mathbb{C}\setminus\{0\} \times H^1_0(\Omega)$ is an eigenpair of (\ref{eq:dweakform}), then $(\lambda^{-1},\, u)$ is an eigenpair for the operator $T$.

For the application of IFEM to eigenvalue problems, we construct IFEM with a penalty term. We start by presenting a modified $P_1$-nonconforming IFEM for the elliptic problem (\ref{eq}) - (\ref{BC}).
For some additional notations, let the collection of all the edges of $K\in \mathcal K_h$ be denoted by $\mathcal{E}_h$. We split $\mathcal{E}_h$ into two disjoint sets $\mathcal{E}_h = \mathcal{E}_h^o\cup \mathcal{E}_h^b$,
where $\mathcal{E}_h^o$ is the set of edges lying in the interior of $\Omega$,
and $\mathcal{E}_h^b$ is the set of edges on the boundary of $\Omega$.

The IFEM (modified by a penalty term)  for  (\ref{op}) is to find  $\hat{u}_h \in \what N_h(\Omega)$ such that
\begin{equation}
a_h^\sigma(\hat{u}_h, \hat{\phi}) =  (f, \hat{\phi}), \quad \forall \hat{\phi} \in \what N_h(\Omega),
 \label{eq:dsweakform}
\end{equation}
where
\begin{eqnarray*}
 a_h^\sigma(u,v) &:=&  a_h(u,v) + j_\sigma(u,v),\\
 a_h(u,v) &:=&  \sum_{K\in \mathcal{K}_h}\int_K \beta \nabla u \cdot \nabla v\, dx, \\
 j_\sigma(u,v) &:=& \sum_{e\in \mathcal E^o_h} \int_{e}\frac \sigma h [u]_{e}[v]_{e}\, ds,
 \mbox { for some } \sigma>0.
\end{eqnarray*}
We define the mesh dependent norm $\|\cdot\|_{1,J}$ on the space $H_h(\Omega)$ by
$$\|v\|^2_{1,J} := \sum_{K \in \mathcal{K}_h} \|v\|^2_{0,K} + \sum_{K \in \mathcal{K}_h} \|\nabla u\|^2_{0,K} + \sum_{e \in \mathcal E^o_h} h^{-1}\|[v]\|^2_{0,e}.$$
By the trace inequality \cite{Brenner-Scott}, this norm is equivalent to  $\|\cdot\|_{1,h}$. The coerciveness and  boundedness of
the bilinear form $ a_h^\sigma(\cdot,\cdot)$ are satisfied.
\begin{lemma}
There exist positive constants $C_b$ and $C_c$ such that
\begin{alignat*}{2}
|a_h^{\sigma}(u,v)| &\leq C_b\|u\|_{1,J}\|v\|_{1,J}, \quad &\forall\, u,v \in H_h(\Omega), \\
a_h^{\sigma}(v,v) &\geq  C_c\|v\|^2_{1,J}, 	 &\forall\, v \in \what N_h(\Omega).
\end{alignat*}
\end{lemma}
The following error estimate for
(\ref{eq:dsweakform}) can be obtained by a slight modification of the proof  in \cite{Kwak-W-C}, by noting that
$j_\sigma(u,v) =0$ for any $u\in {H}^1(\Omega)$ and  $v\in \what N_h(\Omega)$.
\begin{theorem} \label{thm:energyerror}
Let $u\in \wtilde{H}^2(\Omega),~  \hat u_h\in \what N_h(\Omega)$ be the solutions of
(\ref{op}) and (\ref{eq:dsweakform}), respectively. Then there exists a constant $C>0$ such
that
\begin{eqnarray*}
\|u- \hat u_h\|_{0,\Omega}+h\|u- \hat u_h\|_{1,J}\leq C h^2\|u\|_{\wtilde{H}^2(\Omega)}.
\end{eqnarray*}
\end{theorem}

 The IFEM for the eigenvalue problem (\ref{eq:modelEq}) is to find the pairs $(\lambda_h,\hat{u}_h) \in \mathbb{C} \times \what{N}_h(\Omega)$ such that
 \begin{equation}
a_h^\sigma(\hat{u}_h, \hat{\phi}) =  \lambda_h(\hat{u}_h, \hat{\phi}), \quad \forall \hat{\phi}
\in \what N_h(\Omega).
 \nonumber
\end{equation}
Let us define the discrete solution operator $T_h:L^2(\Omega) \rightarrow \what N_h(\Omega)$ by \begin{equation}
a_h^\sigma(T_h f, \hat\phi) = (f, \hat\phi), \quad \forall \hat\phi \in \what N_h(\Omega)
\label{eq:revariform} \nonumber
\end{equation}
with $f \in L^2(\Omega)$. In view of the definition of the discrete solution operator $T_h$, the eigenvalues $\mu_h$ of the operator $T_h$ are given by $\mu_h =1/ \lambda_h$.

\section{Spectral approximation}
Now we are concerned with the spectral approximation that can be proved by using some properties of compact operators in Banach space. We follow the approaches given in \cite{Descloux-Nassif-Rappaz1978-1,Descloux-Nassif-Rappaz1978-2}.

Clearly, the operator $T$ is self-adjoint and bounded. Similarly the operator $T_h$ is self-adjoint such that
\begin{equation}
a_h^\sigma(T_hf,\phi) =  a_h^\sigma(f,T_h\phi), \quad \forall f,\phi \in \what N_h(\Omega). \nonumber
\end{equation}
Next, the boundedness of the operator $T_h$ can be shown by using the coerciveness of the bilinear form $a_h^\sigma(\cdot, \cdot)$. For any $f\in L^2(\Omega)$, it holds that
\begin{align*}
\|T_h f\|^2_{1,J} &\leq Ca_h^\sigma(T_h f,T_h f) \\
		&= C(f,T_h f) \\
		& \leq C\|f\|_{0,\Omega}\|T_h f\|_{0,\Omega} \\
		&\leq C\|f\|_{0,\Omega}\|T_h f\|_{1,J}.
\end{align*}
Therefore, $\|T_h f\|_{1,J} \leq C\|f\|_{0,\Omega} $.

The operator $T$ is compact in $H^1_0(\Omega)$ due to the boundedness of $T$ and Rellich-Kondrachov theorem i.e. the compact embedding $H^1_0(\Omega) \subset L^2(\Omega)$ \cite{Adams-Fournier}. Clearly, the operator $T_h$ is compact in $H_h(\Omega)$ by the definition of $T_h$.

Let $\sigma(T)$ and $\rho(T)$ be the spectrum and resolvent set of $T$, respectively. The spectrum $\sigma(T)$ is a countable set with no accumulation points different from zero and consists of positive real eigenvalues with finite multiplicity. The algebraic multiplicity of each eigenvalue $\mu \in \sigma(T)$ is equal to the geometric multiplicity due to the self-adjointness and compactness of the operator $T$ \cite{Kato}. For any $z\in \rho(T)$, the resolvent operator $R_z(T)$ is defined by $R_z(T) = (z-T)^{-1}$ from $L^2(\Omega)$ to $L^2(\Omega)$ or from $H^1_0(\Omega)$ to $H^1_0(\Omega)$. Following the references \cite{Descloux-Nassif-Rappaz1978-1,Descloux-Nassif-Rappaz1978-2}, we prove the non-pollution of the spectrum $\sigma(T)$. To do so, we need the following results.
\begin{lemma}
For $z\in \rho(T)$, $z\neq 0$, there is a constant $C>0$ depending on only $\Omega$ and $|z|$ such that
$$\|(z-T)f\|_{1,J} \geq C\|f\|_{1,J}, \quad \forall f \in H_h(\Omega).$$
\label{lem:contiresol}
\end{lemma}
\begin{proof}
Let $g = (z- T)f$. We need to show $\|f\|_{1,J} \leq C\|g\|_{1,J}$. From the definition of $T$ and $g$, we have the equalities,
\begin{equation}
a(Tf, v) = a(zf - g, v) = (f,v), \quad \forall v\in H^1_0(\Omega).
\label{eq:lemweak1}
\end{equation}
Reformulating the second equality in (\ref{eq:lemweak1}), we obtain
\begin{equation}
a(zf-g,v) - \frac{1}{z}(zf-g,v) = \frac{1}{z}(g,v), \quad \forall v\in H^1_0(\Omega) .
\label{eq:lemweak}
\end{equation}
Since $z \in \rho(T)$, the inverse $z^{-1}$ is not an eigenvalue of $a(\cdot, \cdot)$. Hence $zf -g$ is the solution of the weak formulation (\ref{eq:lemweak}). By using Theorem \ref{thm:reg}, we have
\begin{equation}
\|zf - g\|_{1,J} \leq C\frac{1}{|z|}\|g\|_{0,\Omega} \leq C\frac{1}{|z|}\|g\|_{1,J}.
\label{eq:contiresol_eq1}
\end{equation}
From the triangle inequality and (\ref{eq:contiresol_eq1}), it follows immediately that
\begin{align*}
\|f\|_{1,J} &\leq \frac{1}{|z|}(\|zf - g\|_{1,J} + \|g\|_{1,J}) \\
 &\leq \frac{1}{|z|}(C\frac{1}{|z|}\|g\|_{1,J} + \|g\|_{1,J}) \\
 & \leq C(|z|) \|g\|_{1,J}\; ,
\end{align*}
where $C(|z|)$ is a constant depending on $|z|$.
\end{proof}
\newline

\begin{theorem}
For $z\in \rho(T)$, $z\neq 0$, there is a constant $C>0$ depending only on $\Omega$ and $|z|$ such that for $h$ small enough
$$\|(z-T_h)f\|_{1,J} \geq C\|f\|_{1,J}, \quad \forall f\in H_h(\Omega).$$
In other words, the resolvent operator $R_z(T_h) = (z-T_h)^{-1}$ is bounded.
\label{thm:disResol}
\end{theorem}
\begin{proof}
By Theorem \ref{thm:energyerror} and Lemma \ref{lem:contiresol}, we get
\begin{align*}
\|(z-T_h)f\|_{1,J} &\geq \|(z-T)f\|_{1,J} - \|(T-T_h)f\|_{1,J} \\
				&\geq (C_1(|z|) - C_2 h)\|f\|_{1,J} \\
				&\geq C(|z|)\|f\|_{1,J},
\end{align*}
for $h$ small enough.
\end{proof}
\newline

Before we state the following Corollary, we denote an operator norm $\|L\|_{\mathscr{L}(X,Y)}$ for a bounded linear operator $L : X \to Y$ by
\begin{equation}
\|L\|_{\mathscr{L}(X,Y)} = \sup_{x \in X} \frac{\|Lx\|_Y}{\|x\|_X}. \label{eq:oper-norm}
\end{equation}
\begin{corollary}
Let $F \subset \rho(T)$ be closed, then
$$\|R_z(T_h)\|_{\mathscr{L}(H_h(\Omega), H_h(\Omega))} \leq C, \quad \forall z \in F,$$
for some constant $C$.
\label{lem:disResolBdd}
\end{corollary}

The following result is a direct consequence of Corollary \ref{lem:disResolBdd}. We note that the proof is analogous to Theorem 1 in \cite{Descloux-Nassif-Rappaz1978-1}.
\begin{theorem}
(Non-pollution of the spectrum) Let $A \subset \mathbb{C}$ be an open set containing $\sigma(T)$. Then for sufficiently small $h$, $\sigma(T_h) \subset A$.
\label{thm:nonpollspec}
\end{theorem}
\\
This implies that there are no discrete spurious eigenvalues of the solution operator $T_h$.

Now we turn to show the non-pollution and completeness of the eigenspace \cite{Descloux-Nassif-Rappaz1978-1,Descloux-Nassif-Rappaz1978-2}. Let $\mu$ be an eigenvalue of $T$ with algebraic multiplicity $n$. We define the spectral projection $E(\mu)$ from $L^2(\Omega)$ into $H^1_0(\Omega)$ by
$$E(\mu) = \frac{1}{2\pi i}\int_{\Lambda} R_z(T)\, dz, $$
where $\Lambda$ be a Jordan curve in $\mathbb{C}$ containing $\mu$, which lies in $\rho(T)$ and does not enclose any other points of $\sigma(T)$ \cite{Kato}. By Corollary \ref{lem:disResolBdd}, the discrete resolvent operator $R_z (T_h)$ is bounded. Therefore, we can define the discrete spectral projection $E_h(\mu)$ from $L^2(\Omega)$ into $H_h(\Omega)$ by
$$E_h(\mu) = \frac{1}{2\pi i} \int_{\Lambda} R_z(T_h)\, dz. $$
The projections $E(\mu)$ and $E_h(\mu)$ are simply denoted by $E$ and $E_h$, respectively. The following Theorem provides the uniform convergence of spectral projections.
\begin{theorem}
It holds that
\begin{equation}
\lim_{h \to 0}\|E - E_h\|_{\mathscr{L}(L^2(\Omega), H_h(\Omega))} = 0. \nonumber
\end{equation}
\label{thm:specConv}
\end{theorem}
\begin{proof}
By using the resolvent identity
\begin{equation*}
R_z(T) - R_z(T_h) = R_z(T_h)(T-T_h)R_z(T),
\end{equation*}
we obtain for $f \in L^2(\Omega)$,
\begin{equation*}
\begin{split}
\|(E - E_h)f\|_{1,J} & \leq C\|(R_z(T) - R_z(T_h))f\|_{1,J} \\
	& = C\| R_z(T_h)  (T-T_h)  R_z(T) f\|_{1,J} \\
	& \leq C\| R_z(T_h)  \|_{\mathscr{L}(H_h(\Omega), H_h(\Omega))}\|T-T_h\|_{\mathscr{L}(L^2(\Omega), H_h(\Omega))} \\
	&\qquad \cdot \|R_z(T)\|_{\mathscr{L}(L^2(\Omega), L^2(\Omega))}\|f\|_{L^2(\Omega)}.
\end{split}
\end{equation*}
For $h$ small enough, $\|R_z(T_h)\|_{\mathscr{L}(H_h(\Omega), H_h(\Omega))}$ and $\|R_z(T)\|_{\mathscr{L}(L^2(\Omega), L^2(\Omega))}$ are bounded by Theorem \ref{thm:disResol} and Fredholm alternative \cite{Conway}, respectively.
The operator norm $\|T-T_h\|_{\mathscr{L}(L^2(\Omega), H_h(\Omega))}$ goes to zero as $h \to 0$.
The proof is now complete.
\end{proof}

We are now in a position to show the boundedness of the distance between eigenspaces. Such a distance for any closed subspaces of $H_h(\Omega)$ may be evaluated by means  of  distance functions
\begin{eqnarray*}
&\mathrm{dist}_h (x, Y) &\;= \inf_{y \in Y}\|x - y\|_{1,J}, \quad \mathrm{dist}_h(Y,Z) \:= \sup_{y \in Y, \|y\|_{1,J} = 1}\mathrm{dist}_h(y,Z), \\
&\mathrm{dist}\;(Y,\,Z) &\;= \max(\mathrm{dist}_h(Y,Z), \, \mathrm{dist}_h(Z,Y)).
\end{eqnarray*}
The following results are analogous to Theorem \ref{thm:specConv}, whose proofs can be obtained as in \cite{Descloux-Nassif-Rappaz1978-1}.
\begin{theorem}
\begin{itemize}
\item (Non-pollution of the eigenspace) $$\lim_{h \to 0} \mathrm{dist}_h(E_h(H_h(\Omega)), E(H^1_0(\Omega)))=0.$$
\item (Completeness of the eigenspace) $$\lim_{h \to 0} \mathrm{dist}_h(E(H^1_0(\Omega)), E_h(H_h(\Omega)))=0.$$
\label{n.p.e}
\end{itemize}
\end{theorem}

It remains to show that the distance between the spectrums of $T$ and $T_h$ vanishes as $h$ goes to zero.
\begin{theorem}
(Completeness of the spectrum) For all $z \in \sigma(T)$,
$$\lim_{h \to 0} \mathrm{dist}_h (z, \sigma(T_h)) = 0.$$
\label{c.s}
\end{theorem}
\begin{proof}
The proof follows from Theorem $6$ in \cite{Descloux-Nassif-Rappaz1978-1}.
\end{proof}

\section{Convergence analysis}
In this section, we present the convergence analysis of eigenvalues. By using the spectral properties of compact operators in the previous section, we show the convergence rate of eigenvalues.
\begin{theorem}\label{thm:convEig}
Let $\mu$ be an eigenvalue of $T$ with multiplicity $n$. Then for $h$ small enough there exist $n$ eigenvalues $\{\mu_{1,h}, ... , \mu_{n,h}\}$ of $T_h$ which converge to $\mu$ as follows
\begin{equation*}
\sup_{1\leq i \leq n}|\mu - \mu_{i,h}| \leq C h^2,
\end{equation*}
where a positive constant $C$ is independent of $\mu$ and $h$.
\end{theorem}
\begin{proof}
The existence of $\mu_{i,h}$ is a direct consequence of the previous section. Now we estimate the convergence rate of $\mu_{i,h}$.
Let $\Phi_h$ be the restriction of $E_h$ to $E(L^2(\Omega))$:
$$\Phi_h=E_h|_{E(L^2(\Omega))} : E(L^2(\Omega)) \to E_h(H_h(\Omega)).$$
Following the arguments in \cite{Babuska-Osborn, Osborn}, we can show that the inverse $\Phi_h^{-1} : E_h(H_h(\Omega)) \to E(L^2(\Omega))$ is bounded for $h$ small enough. To show $\Phi_h^{-1}$ is defined, let $\Phi_h f = 0$ with $f \in E(L^2(\Omega))$. Then by Theorem \ref{thm:specConv}, we have
$$\|f\|_{0,\Omega} = \|f - \Phi_h f\|_{0,\Omega} = \|E f - E_h f\|_{0,\Omega} \leq \|E - E_h\|_{\mathscr{L}(L^2(\Omega), H_h(\Omega))}\|f\|_{0,\Omega}.$$
Thus $\Phi_h$ is one-to-one. By Theorem \ref{n.p.e}, $\Phi_h$ is onto such that the inverse $\Phi_h^{-1}$ is defined. Now we show that $\Phi_h^{-1}$ is bounded. For $f \in E(L^2(\Omega))$ and $h$ small enough,
\begin{align*}
\|\Phi_h f \|_{0,\Omega} &\geq \|f\|_{0,\Omega} - \|f - \Phi_h f \|_{0,\Omega} \\
	&= \|f\|_{0,\Omega} - \|E f - E_h f\|_{0,\Omega} \\
	&\geq (1 - \|E- E_h\|_{\mathscr{L}(L^2(\Omega), H_h(\Omega))}) \|f\|_{0,\Omega} \\
	&\geq \frac{1}{2} \|f\|_{0,\Omega}.
\end{align*}
Hence the inverse $\Phi_h^{-1}$ is bounded.

Let $\wtilde T$ be the restriction of $T$ to $E(L^2(\Omega))$ and define $\wtilde T_h := \Phi^{-1}_h T_h \Phi_h$. Setting $S_h = \Phi_h^{-1}E_h : L^2(\Omega) \to E(L^2(\Omega))$ (see Figure \ref{fig:operator}), we see that $S_h$ is bounded and $S_h f = f$ for any $f \in E(L^2(\Omega))$. By definitions of $\wtilde T$, $\wtilde T_h$, $S_h$ and $\Phi_h$, and the property $T_h E_h = E_h T_h$, we have for any $f \in E(L^2(\Omega))$,
\begin{align}
(\wtilde{T}-\wtilde{T}_h)f &= Tf - \Phi_h^{-1}T_h \Phi_h f \nonumber \\
	&= S_h T f - \Phi_h^{-1} T_h E_h f \nonumber \\
	&= S_h T f - \Phi_h^{-1}E_hT_h f \nonumber  \\
	&= S_h(T - T_h)f. \nonumber
\end{align}
By definition of operator norm (\ref{eq:oper-norm}) and Theorem \ref{thm:energyerror}, we have
\begin{align*}
\sup_{1\leq i \leq n} |\mu - \mu_{i,h}| &\leq C \| \wtilde T - \wtilde T_h\|_{\mathscr{L}(E(L^2(\Omega)), E(L^2(\Omega)))} \\
	&= C \sup_{f \in E(L^2(\Omega))} \frac{\|(\wtilde T - \wtilde T_h)f\|_{0,\Omega}}{\|u\|_{0,\Omega}} \\
	&= C \sup_{f \in E(L^2(\Omega))} \frac{\|S_h(T - T_h) f\|_{0,\Omega}}{\|f \|_{0,\Omega}} \\
	&\leq C \sup_{f \in E(L^2(\Omega))} \frac{\|(T - T_h) f\|_{0, \Omega}}{\|f \|_{0,\Omega}} \\
	&\leq C h^2\;.
\end{align*}
\begin{figure}[!ht]
\begin{center}
 \begin{pspicture}(-1.5,-.7)(10.,3.)
 \small
\psline[linewidth=0.5pt]{->}(0.67,2.3) (3.4 ,2.3)
\rput(2,2.6){ $E_h$ }
\psline[linewidth=0.5pt]{<-}(0.2,0.3) (3.6,2.0)
\rput(1.7,1.15){ $\Phi_h^{-1}$}
\rput(0.1,1.3){$S_h=\Phi_h^{-1}E_h$}
\rput(0,2.3){$L^2(\Omega) $}  \rput(4,2.3){$ E_h(H_h)$}
\rput(0,0){$E(L^2(\Omega))$}
\rput(1.14,1.6){\psarcn[linecolor=red,linewidth=0.5pt]{->}(0,0){0.3}{120}{240}}
\psline[linewidth=0.5pt,linestyle=dashed]{->}(0,1.95) (0,0.5)

\rput(5.0,0){
\psline[linewidth=0.5pt]{->}(0.75,0.0) (3.6,0.0)
\rput(2,-0.3){ $\Phi_h$}
\psline[linewidth=0.5pt]{<-}(0.2,0.3) (3.6,2.0)
\rput(1.7,1.15){ $\Phi_h^{-1}$}
 \rput(4,2.3){$ E_h(H_h)$}
\rput(0,0){$E(L^2(\Omega))$}  \rput(4.2,0){$ E_h(H_h)$}
\psline[linewidth=0.5pt]{<-}(4,1.95) (4,0.5)
\rput(0.2,0.1) {\psarc[linecolor=blue,linewidth=0.5pt,linestyle=dashed]{->}(0,0){0.45}{250}{160}} 
\rput(2.65,0.6){$\tilde T_h=\Phi_h^{-1}T_h\Phi_h $}
\rput(4.3,1.3){$T_h $}
}
\end{pspicture}
\end{center}
\caption{The operators $S_h$ and $\tilde T_h$}
\label{fig:operator}
\end{figure}
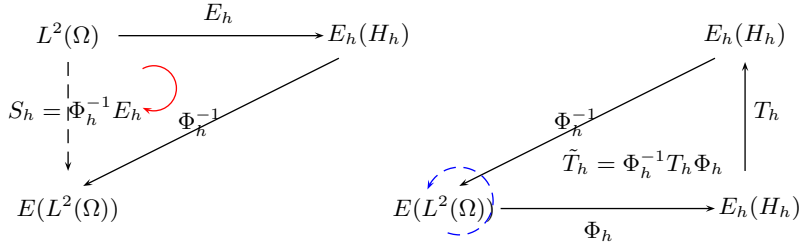
\end{proof}
\begin{remark}
Theorem \ref{thm:convEig} can be expressed in terms of the eigenvalues $\lambda = \mu^{-1}$ and $\lambda_{i,h} = \mu_{i,h}^{-1}$ as
$$\frac{|\lambda - \lambda_{i,h}|}{|\lambda_{i,h}|} \leq C_1 h^2.$$
For $h$ small enough, we derive the estimate for the relative error
$$\frac{|\lambda - \lambda_{i,h}|}{|\lambda|} \leq \frac{C_1 h^2}{1 - C_1|\lambda| h^2} \leq C h^2.$$
\end{remark}

\section{Numerical results}
We demonstrate numerical experiments for the problem (\ref{eq:modelEq}). In the first example, we test an elliptic eigenvalue problem with a circular interface for which we know the exact eigenvalues. Next we perform an experiment for the case with star-shaped interface.  We observe the optimal orders of convergence of numerical eigenvalues. In our computations we use the package ARPACK \cite{Lehoucq-Sorensen-Yang} which is designed for solving large sparse eigenvalue problems.

\textbf{Example 1}. Let a circular computational domain be $\Omega = \{ (r,\theta) : 0 \leq r \leq R_O, 0 \leq \theta < 2\pi\}$ with an interface $\Gamma = \{(r, \theta) : r = R_I, 0 \leq \theta < 2\pi\}$. The eigenpairs $(\lambda,u)$ of the model problem (\ref{eq:modelEq}) are given by $u(x,y) = R(r)\Theta(\theta)$,
\begin{align*}
&\Theta(\theta) = d_1\cos m \theta + d_2\sin m\theta, \\
& R(r) =  \left\{\begin{array}{l l} c_1^{+}J_{m}(\sqrt{\frac{\lambda}{\beta^{+}}}r) + c_2^{+}Y_{m}(\sqrt{\frac{\lambda}{\beta^{+}}}r), & R_I < r \leq R_O, \\
 c_1^{-}J_{m}(\sqrt{\frac{\lambda}{\beta^{-}}}r) , & 0 \leq r \leq R_I,
 \end{array}
 \right.
\end{align*}
where $c^{\pm}_i$ and $d_i$ are constants, and $J_{m}$ and $Y_{m}$ the Bessel functions of the first and second kind of order $m$, respectively.
In Appendix, we explain in more details how the coefficients $c^{\pm}_i$, $d_i$ could be determined.
 We set $R_O=1$, $R_I=0.38$ and $(\beta^-, \beta^+) = (1,1000), (1000,1)$. It seems to be good to choose $\sigma$ dependent on $\beta$, say  $\sigma = \kappa \beta$ for some $ \kappa >0$.  The triangulation of the circular domain consists of qusi-regular triangles with the maximal diameter $h$, which may intersect the interface $\Gamma$ as Figure \ref{fig:mesh1}. Tables \ref{table:circle1-} and \ref{table:circle1+} show the first ten eigenvalues and their rates of convergence. The first columns are the exact values and the other columns are the eigenvalues of IFEM for varying $h$. From the second to sixth column, the meshes are generated so that the degree of freedom quadruples, thus $h$ nearly halves. The numbers in parentheses for each column show the order of convergence. We observe that the order of convergence is quadratic and there are no spurious eigenvalues. Figure \ref{fig:circle-eigV} illustrates two eigenfunctions corresponding to eigenvalues $\lambda_1$ and $\lambda_2$ in the case of $\beta^- = 1, \, \beta^+ = 1000$. The cases with other values of $R_I$ and $\beta$ show similar results, although we do not present here.

We recount the influence of penalty parameter $\sigma$ in (\ref{eq:dsweakform}). The results are shown in Figure \ref{fig:sigma}. We notice that the case for $\kappa = 0.1$ has some deteriorations in the order of convergence. The cases when $\kappa \in [1,100]$ achieve desired convergence orders (see Theorem \ref{thm:convEig}). In Tables \ref{table:circle1-} and \ref{table:circle1+}, we choose $\kappa = 1$.
\begin{figure}[!ht]
	\centering
	 \includegraphics[width=0.5\textwidth, height=0.5\textwidth]{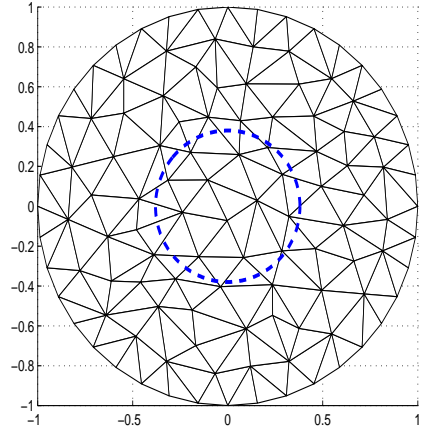}
	\caption{Example of mesh generation. The inner broken line represents the interface $\Gamma$.}
	\label{fig:mesh1}
\end{figure}

\renewcommand{\arraystretch}{1.2}
\begin{table}[!ht] \footnotesize
\centering
\begin{tabular}{|r||r|r|r|r|r|} \hline
\multicolumn{1}{|c||}{$\lambda_{exact}$}	  & \multicolumn{1}{|c}{Level 1 (ord)} & \multicolumn{1}{|c}{Level 2 (ord)}  & \multicolumn{1}{|c}{Level 3 (ord)}  & \multicolumn{1}{|c}{Level 4 (ord)}  & \multicolumn{1}{|c|}{Level 5 (ord)} \\ \hline
39.972&		40.018	(2.11) &	39.982	(2.15) &	39.974	(1.99) &	39.972	(2.08) &	39.972	(1.99) \\ 
101.523&		101.744	(2.40)&	101.566	(2.33) &	101.533	(2.02) &	101.525	(2.12) &	101.523	(2.17) \\ 
101.523&		101.772	(2.52) &	101.573	(2.31) &	101.534	(2.18) &	101.525	(2.13) &	101.523	(2.18) \\ 
182.473	&	183.212	(2.64) &	182.626	(2.27) &	182.507	(2.18) &	182.481	(2.20) &	182.475	(2.25) \\ 
182.473	&	183.522	(2.56) &	182.636	(2.68) &	182.507	(2.27) &	182.481	(2.16) &	182.475	(2.27) \\ 
210.604	&	211.120	(1.82) &	210.723	(2.11) &	210.635	(1.96) &	210.612	(2.00) &	210.606	(2.00) \\ 
281.713	&	283.846 (2.54) &	282.098	(2.46) &	281.792	(2.29) &	281.730	(2.18) &	281.716	(2.28) \\ 
281.713	&	284.413	(2.67) &	282.140	(2.66) &	281.798	(2.32) &	281.731	(2.25) &	281.717	(2.29) \\ 
340.329	&	341.615	(2.00) &	340.625	(2.11) &	340.404	(1.98) &	340.347	(2.06) &	340.333	(2.08) \\ 
340.329	&	341.799	(2.16) &	340.643	(2.22) &	340.405	(2.04) &	340.347	(2.05) &	340.333	(2.09) \\ \hline
\multicolumn{1}{|c||}{D.O.F} & 	\multicolumn{1}{|c}{14744}	&	\multicolumn{1}{|c}{59386}	&	\multicolumn{1}{|c}{232605	} &	 \multicolumn{1}{|c}{932343} & \multicolumn{1}{|c|}{3732735} \\ \hline
\end{tabular}
\caption{Eigenvalues by IFEM in Figure \ref{fig:mesh1} when $\beta^- = 1, \beta^+=1000$ and $\kappa = 1$.}
\label{table:circle1-}
\end{table}

\begin{table}[!ht] \footnotesize
\centering
\begin{tabular}{|r||r|r|r|r|r|} \hline
\multicolumn{1}{|c||}{$\lambda_{exact}$}	  & \multicolumn{1}{|c}{Level 1 (ord)} & \multicolumn{1}{|c}{Level 2 (ord)}  & \multicolumn{1}{|c}{Level 3 (ord)}  & \multicolumn{1}{|c}{Level 4 (ord)}  & \multicolumn{1}{|c|}{Level 5 (ord)} \\ \hline
6.047	&	6.049	(2.01) &	6.047	(1.99) &	6.047	(2.00) &	6.047	(1.98) &	6.047	(1.98) \\ 
27.355	&	27.380	(2.35) &	27.360	(2.33) &	27.356	(2.19) &	27.355	(2.13) &	27.355	(2.46) \\ 
27.355	&	27.382	(2.42) &	27.360	(2.36) &	27.356	(2.26) &	27.355	(2.12) &	27.355	(2.46) \\ 
34.126	&	34.175	(2.49) &	34.135	(2.44) &	34.128	(2.24) &	34.126	(2.13) &	34.126	(2.33) \\ 
34.126	&	34.183	(2.42) &	34.136	(2.54) &	34.128	(2.28) &	34.126	(2.16) &	34.126	(2.35) \\ 
39.742	&	39.766	(2.08) &	39.748	(1.99) &	39.744	(2.03) &	39.743	(1.96) &	39.742	(2.01) \\ 
45.091	&	45.171	(2.12) &	45.104	(2.54) &	45.094	(2.21) &	45.091	(2.11) &	45.091	(2.23) \\ 
45.091	&	45.176	(2.62) &	45.106	(2.44) &	45.094	(2.27) &	45.091	(2.17) &	45.091	(2.27) \\ 
59.871	&	59.968	(2.40) &	59.890	(2.31) &	59.875	(2.17) &	59.872	(2.09) &	59.871	(2.17) \\ 
59.871	&	59.990	(2.21) &	59.892	(2.50) &	59.875	(2.20) &	59.872	(2.14) &	59.871	(2.17) \\ \hline
\multicolumn{1}{|c||}{D.O.F} & 	\multicolumn{1}{|c}{14744}	&	\multicolumn{1}{|c}{59386}	&	\multicolumn{1}{|c}{232605	} &	 \multicolumn{1}{|c}{932343} & \multicolumn{1}{|c|}{3732735} \\ \hline
\end{tabular}
\caption{Eigenvalues by IFEM in Figure \ref{fig:mesh1} when $\beta^- = 1000, \beta^+=1$ and $\kappa = 1$. }
\label{table:circle1+}
\end{table}

\begin{figure}[!ht]
	\centering
 	 \includegraphics[width=0.48\textwidth, height=0.48\textwidth]{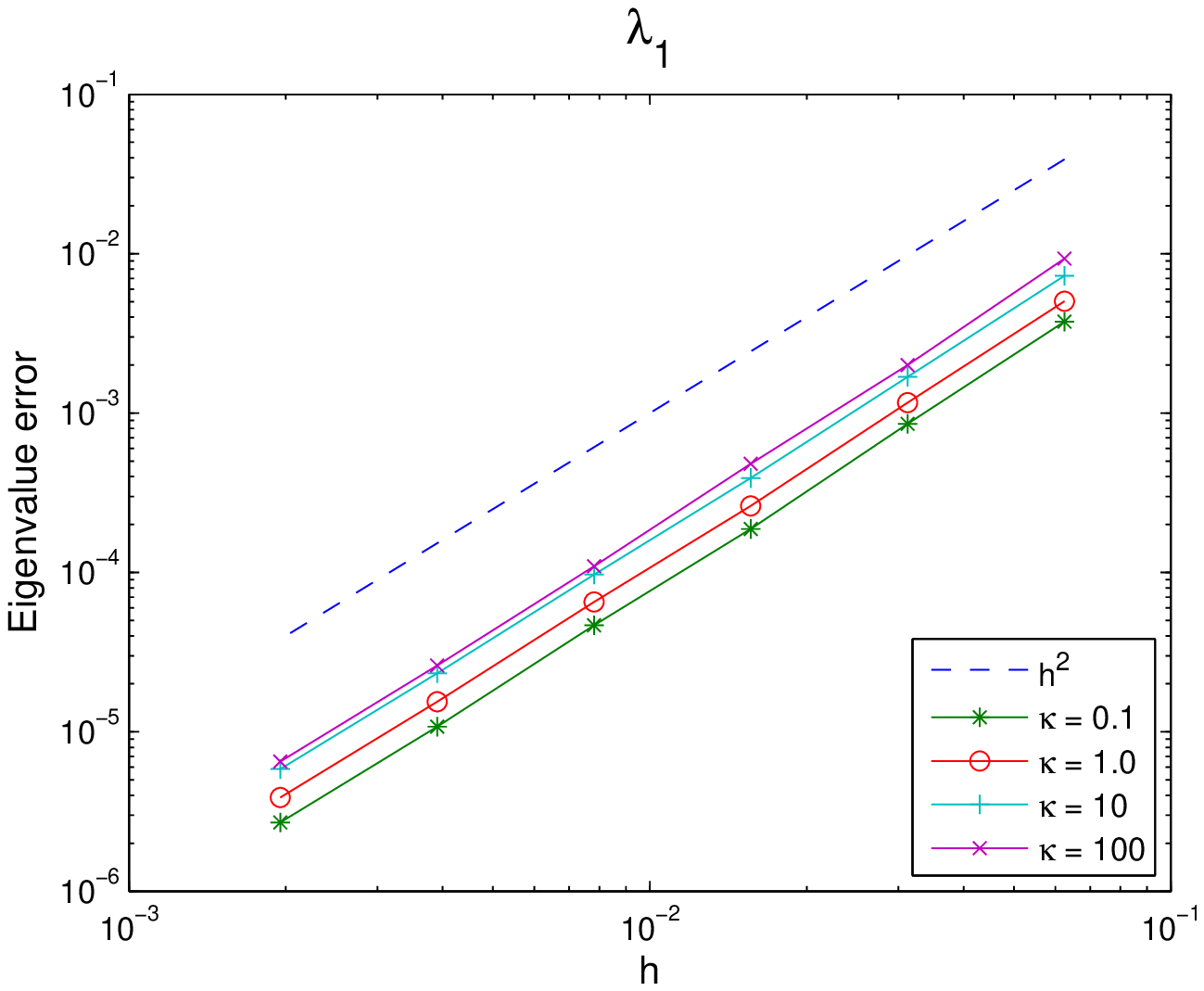}
 	 \includegraphics[width=0.48\textwidth, height=0.48\textwidth]{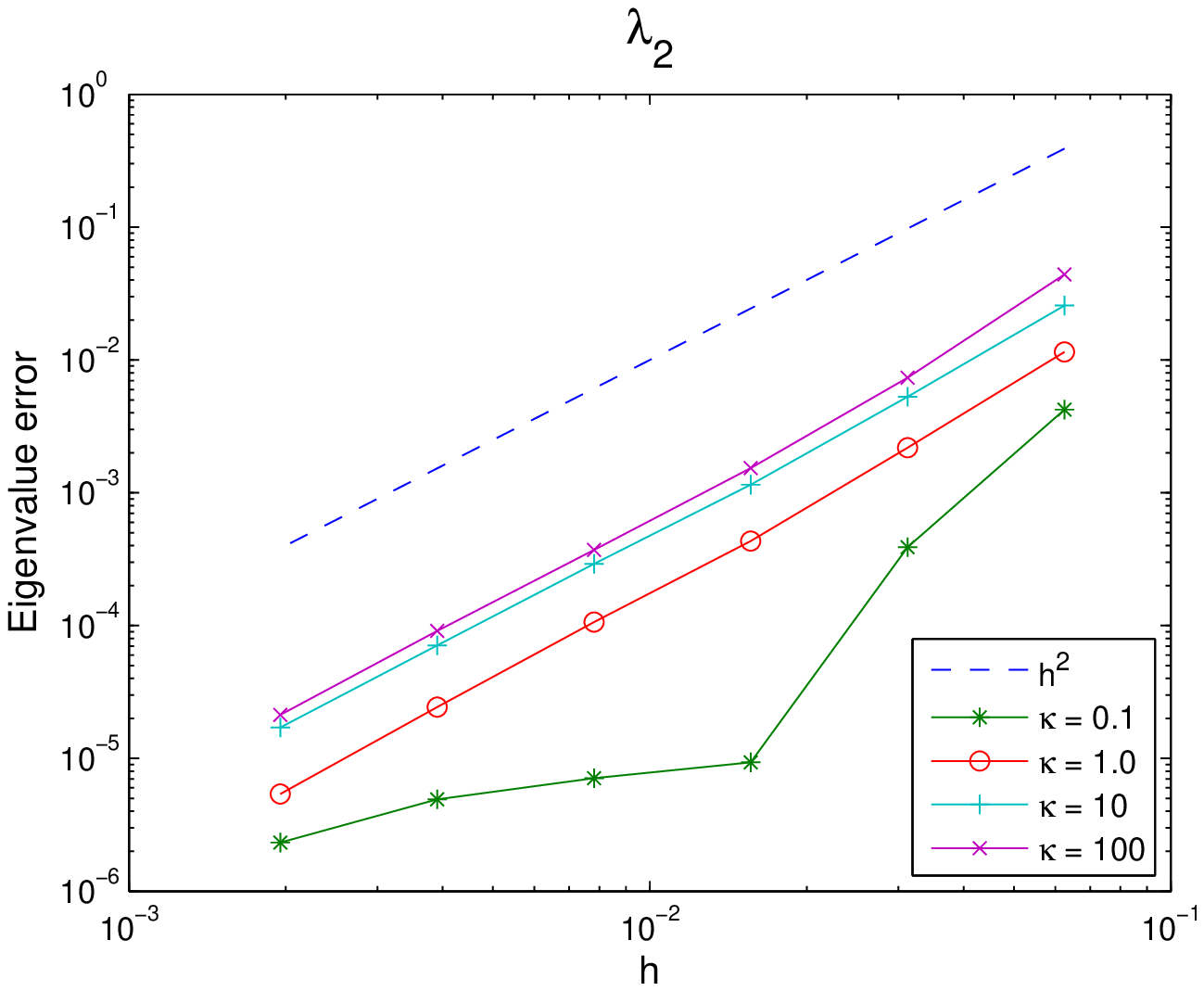}
 	 \includegraphics[width=0.48\textwidth, height=0.48\textwidth]{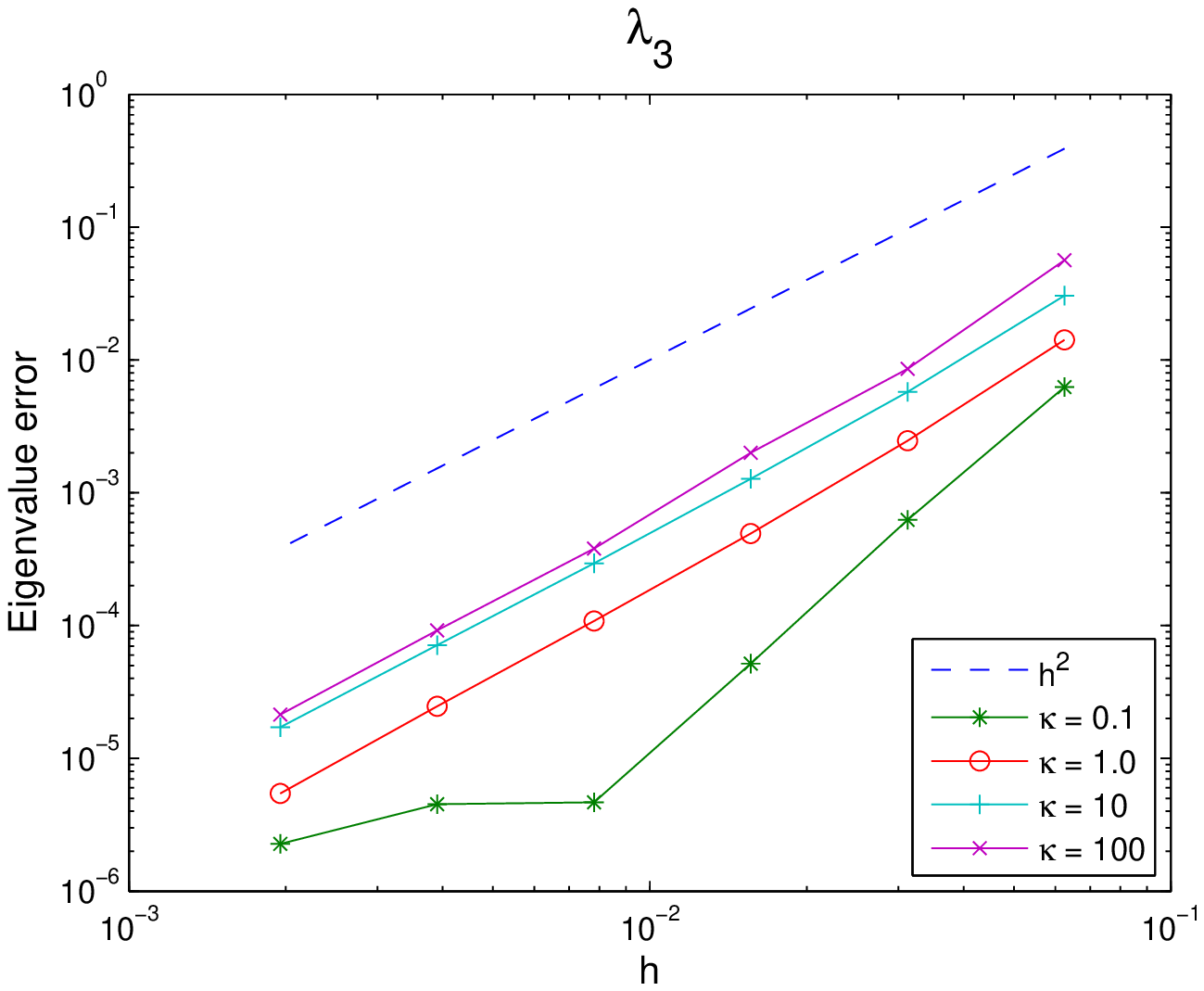}	 	
 	 \includegraphics[width=0.48\textwidth, height=0.48\textwidth]{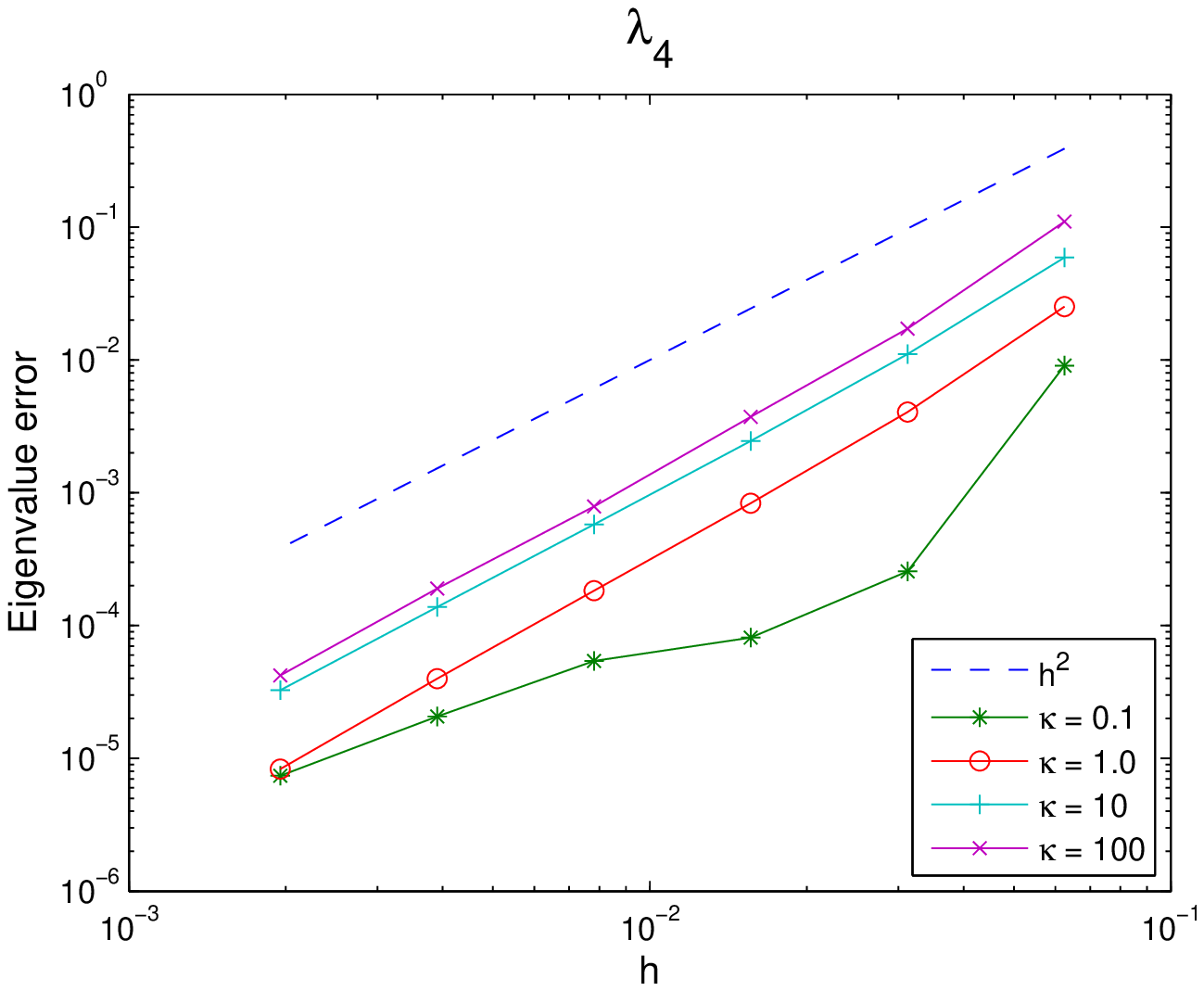}
	\caption{The log-log plot of $h$ versus the relative error of eigenvalues for $\lambda_i,\; 1\leq i \leq 4$ with $\kappa = 0.1$ (asterisk), $\kappa = 1.0$ (circle), $\kappa = 10$ (plus sign) and $\kappa = 100$ (cross). The broken line represents the convergence rate.}
	\label{fig:sigma}
\end{figure}

\begin{figure}[!ht]
	\centering
 	 \includegraphics[width=0.48\textwidth, height=0.48\textwidth]{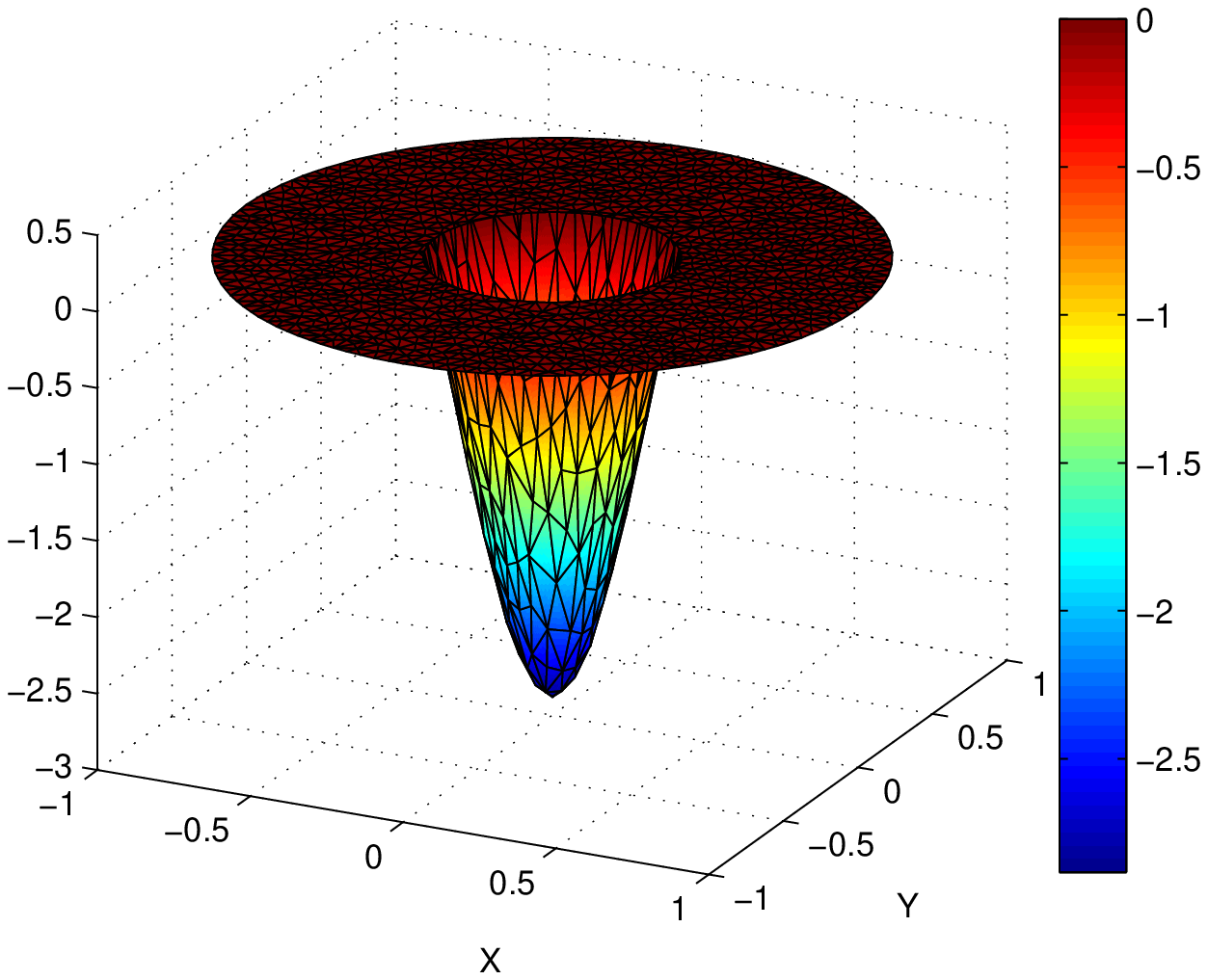}
 	 \includegraphics[width=0.48\textwidth, height=0.48\textwidth]{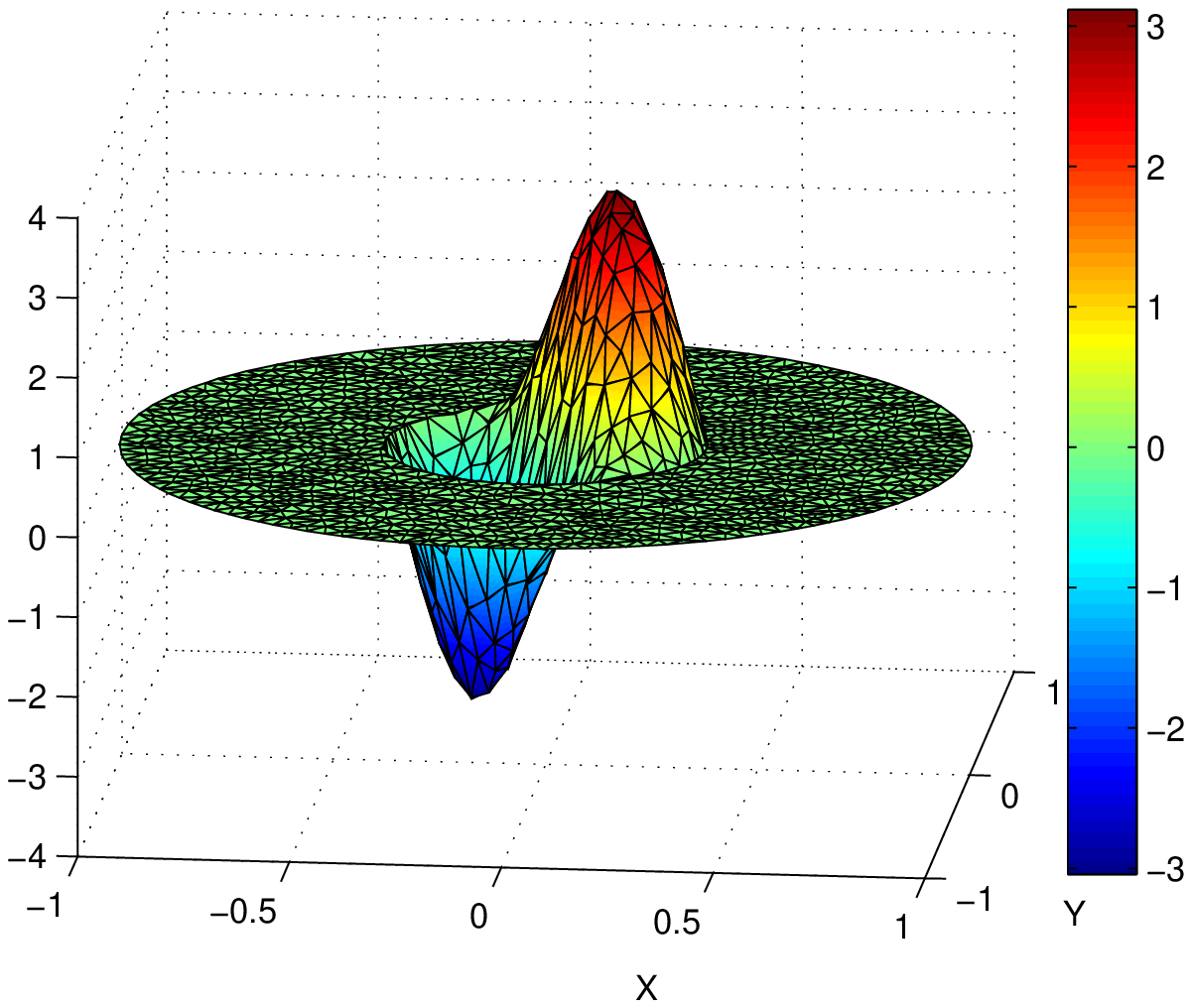}	 	
	\caption{Eigenfunctions corresponding to eigenvalues $\lambda_1$ and $\lambda_2$ in {Example 1} in the case of $(\beta^-, \beta^+)=(1,1000)$.}
	\label{fig:circle-eigV}
\end{figure}

\textbf{Example 2}. Let a computational domain be $\Omega = [-1.1]^2$ and a star-shaped interface is given by $\Gamma = \{(x,y) : \sqrt{x^2+y^2} - 0.2\sin(5\theta - \pi/5) + 0.5 =0\}$, where $\theta = \tan^{-1}(y/x)$. Our computation is performed on a uniform mesh in Figure \ref{fig:mesh2}. Since the exact eigenvalues are not available, we use the numerical results on a sufficiently refined mesh with mesh size $h = 2^{-10}$ as the reference eigenvalues for the purpose of estimating the orders.
Tables \ref{table:star1-} and \ref{table:star1+} contain errors of the eigenvalues $\lambda_h$ with various mesh size $h$ for the interface problem with the coefficient $(\beta^-, \beta^+) = (1,1000) , (1000,1)$. We display some eigenfunctions in Figure \ref{fig:star-eigV}.
\begin{figure}[!ht]
	\centering
	 \includegraphics[width=0.45\textwidth, height=0.45\textwidth]{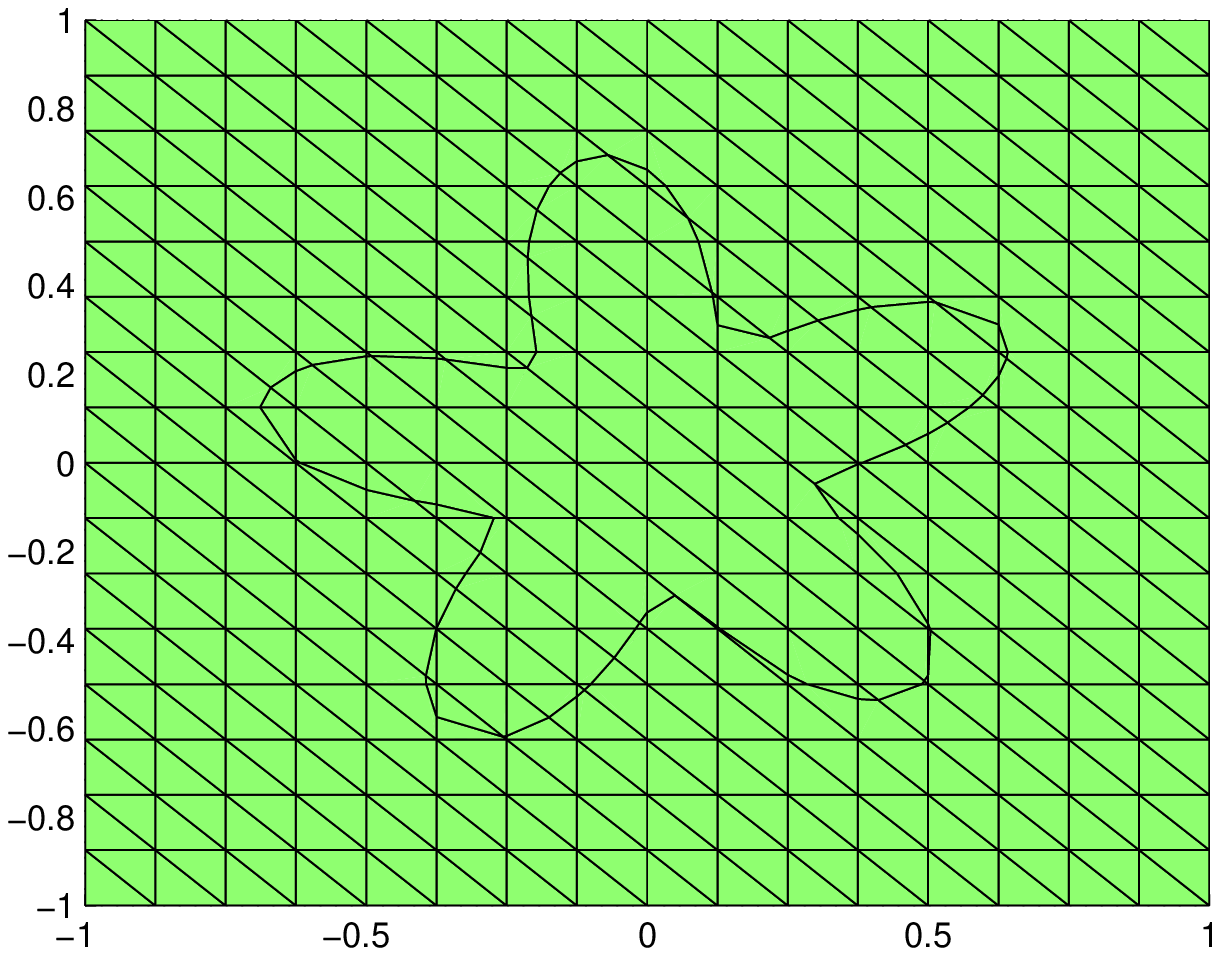}
	\includegraphics[width=0.45\textwidth, height=0.45\textwidth]{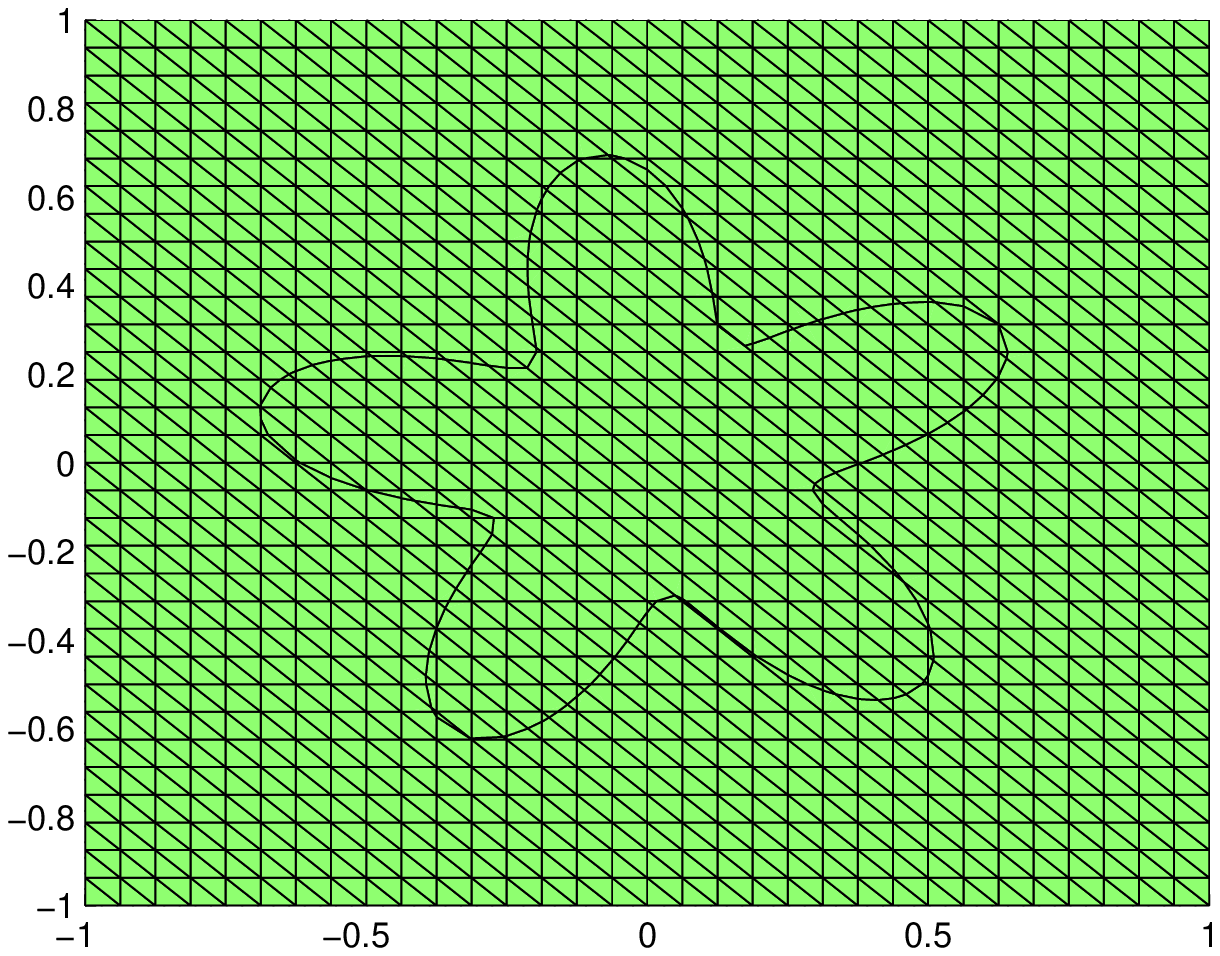}
	\caption{Star-shaped interface with $h=1/2^{3}$ and $h= 1/2^{4}$.}
	\label{fig:mesh2}
\end{figure}

\begin{table}[!ht] \footnotesize
\centering
\begin{tabular}{|c||r|r|r|r|r|} \hline
$\lambda_{ref}$	  & \multicolumn{1}{|c}{$h=1/2^{4}$(ord)} & \multicolumn{1}{|c}{$h=1/2^{5}$(ord)}  & \multicolumn{1}{|c}{$h=1/2^{6}$(ord)}  & \multicolumn{1}{|c}{$h=1/2^{7}$(ord)}  & \multicolumn{1}{|c|}{$h=1/2^{8}$(ord)} \\ \hline
43.206 &	46.794	(2.14)	&	44.313	(1.70) &	43.465	(2.09) &	43.253	(2.46) &	43.218	(1.88) 	 \\
97.442 &	103.736	(1.82)	&	99.028	(1.99) &	97.805	(2.12) &	97.531	(2.02) &	97.461	(2.17) 	 \\
97.442 &	105.890	(2.01)	&	100.338	(1.55) &	98.030	(2.30) &	97.553	(2.40) &	97.472	(1.88) 	 \\
128.947 & 139.100	(1.74)	&	131.402	(2.05) &	129.423	(2.36) &	129.061	(2.05) &	128.972	(2.15) 	 \\
128.955 & 141.697	(2.10)	&	132.611	(1.80) &	129.465	(2.84) &	129.069	(2.16) &	128.980	(2.14) 	 \\
144.481 & 153.170	(2.67)	&	146.479	(2.12) &	144.916	(2.19) &	144.583	(2.09) &	144.503	(2.17) 	 \\
172.374 & 187.286	(2.39)	&	175.982	(2.04) &	173.131	(2.25) &	172.545	(2.14) &	172.412	(2.15) 	 \\
172.374 & 190.745 (2.47)	&	176.841	(2.04) &	173.385	(2.14) &	172.564	(2.41) &	172.420	(2.04) 	 \\
219.650 & 247.605	(2.09)	&	226.494	(2.03) &	220.993	(2.35) &	219.963	(2.10) &	219.723	(2.10) 	 \\
219.652 & 248.667	(2.45)	&	227.195	(1.94) &	221.279	(2.21) &	219.977	(2.32) &	219.728	(2.09) 	 \\ \hline
\end{tabular}
\caption{First ten eigenvalues by IFEM in Figure \ref{fig:mesh2} in the case of $\beta^- = 1, \beta^+=1000$. The reference eigenvalues $\lambda_{ref}$ in the first column are computed with $h = 1/2^{10}$. The numbers in parentheses show convergence rates.}
\label{table:star1-}
\end{table}

\begin{table}[!ht] \footnotesize
\centering
\begin{tabular}{|c||r|r|r|r|r|} \hline
$\lambda_{ref}$  & \multicolumn{1}{|c}{$h=1/2^{4}$(ord)} & \multicolumn{1}{|c}{$h=1/2^{5}$(ord)}  & \multicolumn{1}{|c}{$h=1/2^{6}$(ord)}  & \multicolumn{1}{|c}{$h=1/2^{7}$(ord)}  & \multicolumn{1}{|c|}{$h=1/2^{8}$(ord)} \\ \hline 											
6.052 &	6.096	(1.91)	& 6.062	(2.14) &	6.054	(2.02) &	6.052	(2.49) &	6.052	(2.38)  \\
30.527 &	31.250	(1.79)	& 30.677	(2.26) &	30.567	(1.88) &	30.534	(2.44) &	30.528	(2.21)   \\
32.388 &	32.862	(2.09)	& 32.498	(2.11) &	32.413	(2.14) &	32.393	(2.32) &	32.389	(2.61)   \\
34.867 &	35.532	(1.80)	& 35.035	(1.98) &	34.905	(2.12) &	34.875	(2.23) &	34.868	(2.35)   \\
42.751 &	44.057	(1.61)	& 43.073	(2.02) &	42.827	(2.07) &	42.766	(2.30) &	42.754	(2.29)   \\
45.710 &	46.864	(1.61)	& 45.999	(2.00) &	45.770	(2.27) &	45.722	(2.22) &	45.712	(2.33)   \\
54.380 &	56.156	(1.82)	& 54.807	(2.05) &	54.466	(2.30) &	54.399	(2.15) &	54.384	(2.33)   \\
57.901 &	59.472	(1.87)	& 58.286	(2.03) &	57.987	(2.15) &	57.920	(2.13) &	57.904	(2.37)   \\
62.358 &	64.218	(2.00)	& 62.821	(2.00) &	62.462	(2.14) &	62.381	(2.14) &	62.363	(2.21)   \\
66.220 &	69.027	(1.58)	& 66.831	(2.20) &	66.384	(1.89) &	66.249	(2.49) &	66.226	(2.18)   \\ \hline
\end{tabular}
\caption{First ten eigenvalues by IFEM in Figure \ref{fig:mesh2} in the case of $\beta^- = 1000, \beta^+=1$. The reference eigenvalues $\lambda_{ref}$ in the first column are computed with $h = 1/2^{10}$. The numbers in parentheses show convergence rates.}
\label{table:star1+}
\end{table}

\begin{figure}[!ht]
	\centering
 	 \includegraphics[width=0.48\textwidth, height=0.48\textwidth]{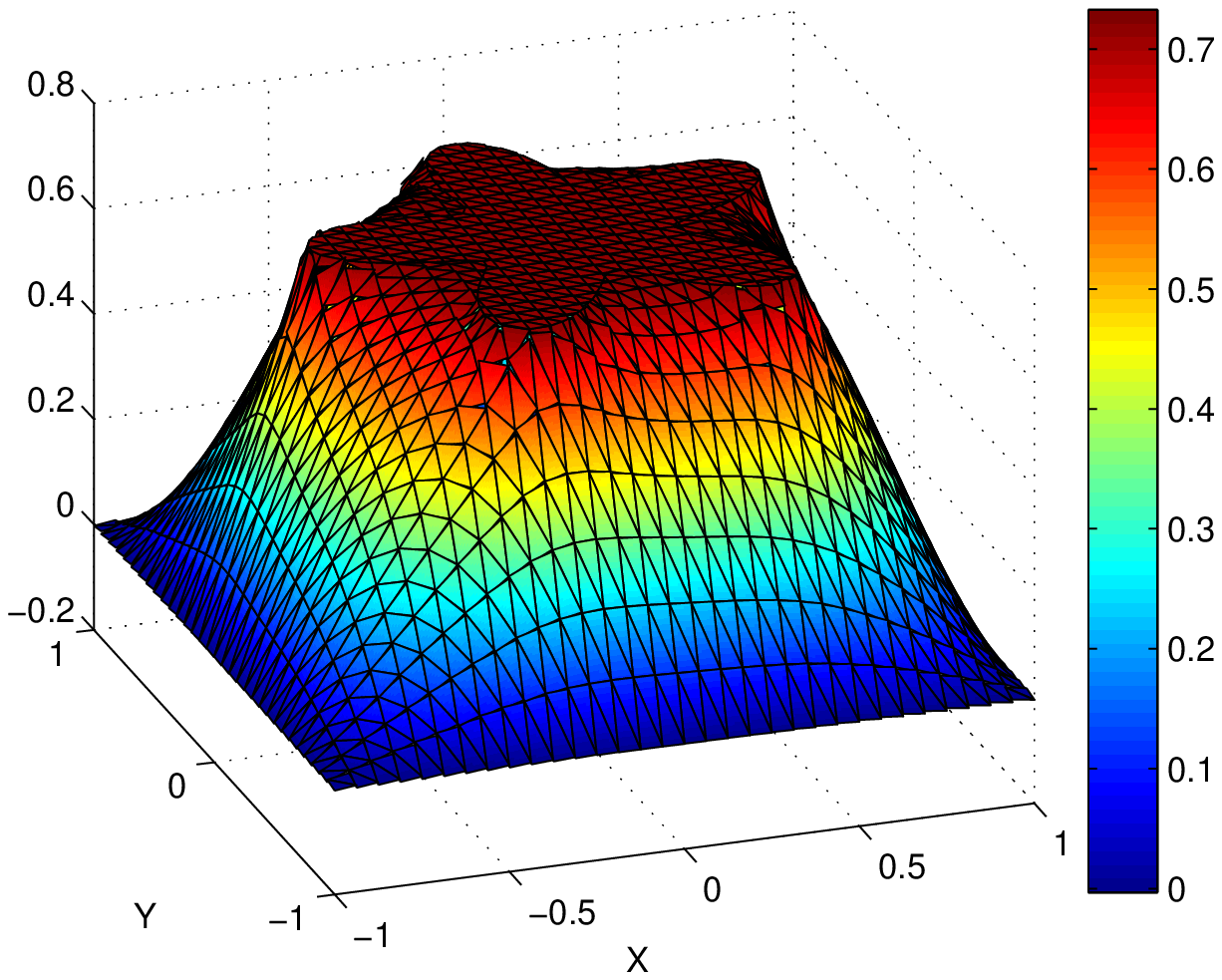}
 	  \includegraphics[width=0.48\textwidth, height=0.48\textwidth]{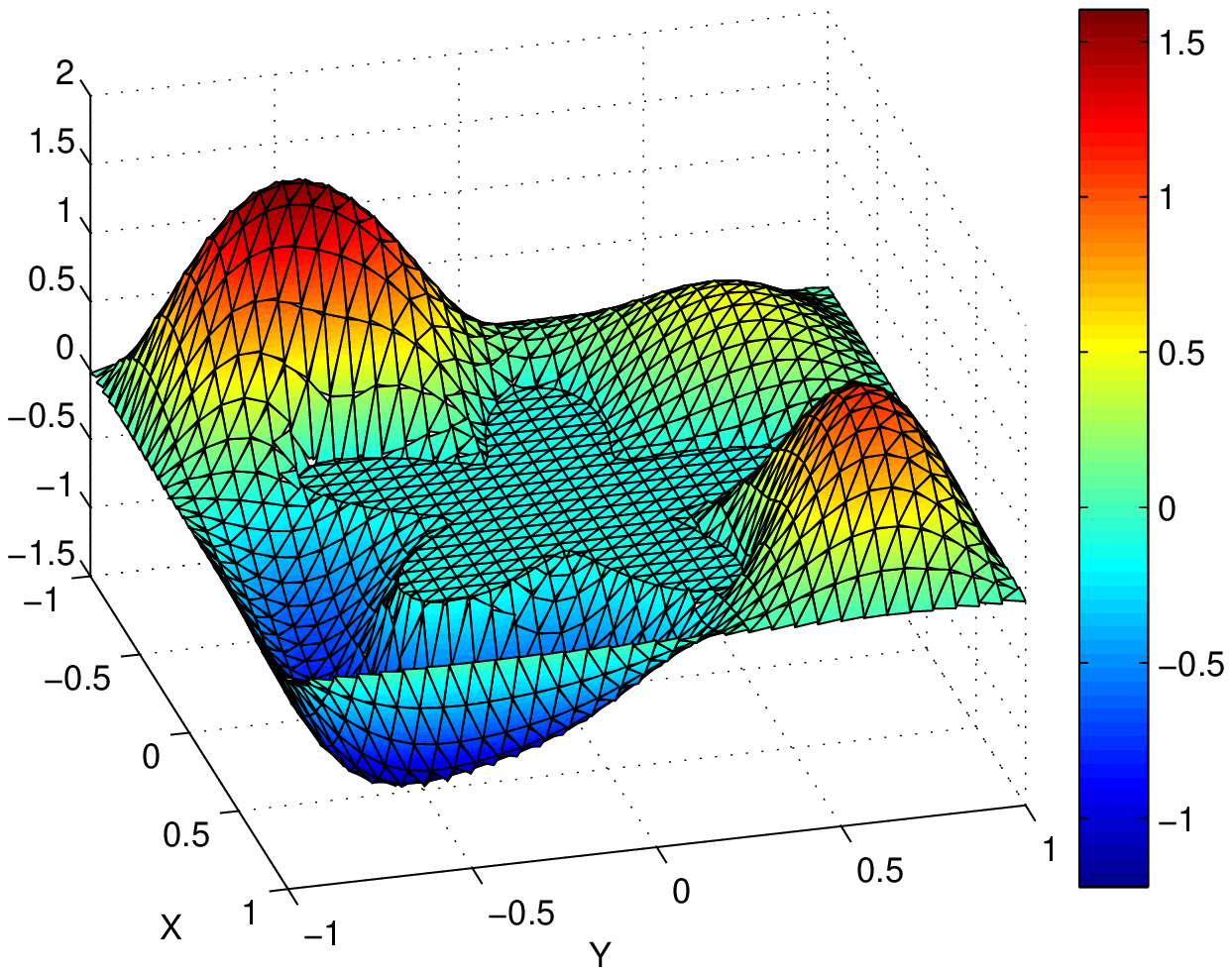}	 	
	\caption{Eigenfunctions corresponding to eigenvalues $\lambda_1$ and $\lambda_4$ in {Example 2} in the case of $(\beta^-, \beta^+)=(1000,1)$.}
	\label{fig:star-eigV}
\end{figure}

\begin{Appendix}
\section*{Appendix}
We show how the eigenvalues from Example 1 in Section 6 can be determined in an analytical way.
Recall the domain $\Omega = \{ (r,\theta) : 0 \leq r \leq R_O, 0 \leq \theta < 2\pi\}$ and the interface $\Gamma = \{(r, \theta) : r = R_I, 0 \leq \theta < 2\pi\}$. The eigenfunction $u(x,y)$ can be determined by the separation of variables, i.e. $u(x,y) = R(r)\Theta(\theta)$. The model problem (\ref{eq:modelEq}) is rewritten in polar coordinates as follows:
\begin{align} \label{eq:model1}
\frac{\partial^2 R}{\partial r^2}\Theta  + \frac{1}{r}\frac{\partial R}{\partial r}\Theta + \frac{1}{r^2} R \frac{\partial^2 \Theta}{\partial \theta^2} &= - \frac{\lambda}{\beta} R\Theta ~~~\text{in}~~ \Omega^s,\quad s=\pm, \\ 
 [R(r)]_{\Gamma} &= 0, ~~~~\left[\beta r \pd {R(r)}{r} \right]_{\Gamma} = 0,  \label{eq:JumpCond} \\
 R(r) &= 0 ~~~~~~\text{on}~~\partial \Omega. \label{eq:BoundCond}
\end{align}
A reformulation of the equation (\ref{eq:model1}) is
\begin{equation}
 \frac{r^2R^{''} + rR^{'} + \frac{\lambda}{\beta}r^2R}{R} = -\frac{\Theta^{''}}{\Theta} = m^2 ~~\text{in}~~\Omega^s,\quad s=\pm. \label{eq:model2}
\end{equation}
 The second relation in (\ref{eq:model2}) gives $\Theta(\theta) = d_1 \cos m\theta + d_2 \sin m\theta$. It also establishes that $m$ is an integer since we must have the same value at $\theta = 0$ and $\theta = 2\pi$.  The first relation in (\ref{eq:model2}) is the Bessel equation
$$r^2 R^{''} + r R^{'} +\left(\frac{\lambda}{\beta} r^2 - m^2\right) R = 0 ~~\text{in}~~\Omega^s, \quad s=\pm.$$
Recall that $\beta$ and $\lambda$ are positive by the properties of the model problem (\ref{eq:modelEq}).
  We obtain $R(r)$ as follows:
 \begin{equation*}
 R(r) =  \left\{\begin{array}{l l} c_1^{+}J_{ m  }\left(\sqrt{\frac{\lambda}{\beta^{+}}}r\right) + c_2^{+}Y_{ m  }\left(\sqrt{\frac{\lambda}{\beta^{+}}}r\right), & \text{in}~\Omega^+, \\
 c_1^{-}J_{ m  }\left(\sqrt{\frac{\lambda}{\beta^{-}}}r\right) + c_2^{-}Y_{ m  }\left(\sqrt{\frac{\lambda}{\beta^{-}}}r\right), & \text{in}~\Omega^-,
 \end{array}\right.
 \end{equation*}
  where $J_{m}$ and $Y_{m}$ are the Bessel functions of the first and second kind of order $m$.
  Since $J_{m}$ is analytic and $Y_{ m  }$ is singular at the origin, we have $c_2^{-} = 0$.
  The coefficients $c_1^+, c_2^+, c_1^-$ are determined by (\ref{eq:JumpCond}) and (\ref{eq:BoundCond}). The condition (\ref{eq:BoundCond}) leads to the equation
 \begin{equation}
 c_1^{+}J_{ m  }\left(\sqrt{\frac{\lambda}{\beta^{+}}}R_O\right) + c_2^{+}Y_{ m  }\left(\sqrt{\frac{\lambda}{\beta^{+}}}R_O\right) = 0 .
 \label{eq:bd1}
 \end{equation}
By using the first relation  of (\ref{eq:JumpCond}), we obtain the equation
\begin{equation}
c_1^{+}J_{ m  }\left(\sqrt{\frac{\lambda}{\beta^{+}}}R_I\right) + c_2^{+}Y_{ m  }\left(\sqrt{\frac{\lambda}{\beta^{+}}}R_I\right) = c_1^{-}J_{ m  }\left(\sqrt{\frac{\lambda}{\beta^{-}}}R_I\right) .
\label{eq:bd2}
\end{equation}
The second part of (\ref{eq:JumpCond}) gives
\begin{eqnarray}
 \, &&\beta^{+}\left(c_1^{+}\frac{d}{d r}\left(J_{ m  }\left(\sqrt{\frac{\lambda}{\beta^{+}}}r\right)\right) + c_2^{+}\frac{d}{d r}\left(Y_{ m  }\left(\sqrt{\frac{\lambda}{\beta^{+}}}r\right)\right)\right) \nonumber\\
 = \, && \beta^{-}c_1^{-}\frac{d}{d r}\left(J_{ m  }\left(\sqrt{\frac{\lambda}{\beta^{-}}}r\right)\right)\;\;\;\;\; \text{on}\;\;\;r = R_I.
\label{eq:bd3}
\end{eqnarray}
From the equations (\ref{eq:bd1}),(\ref{eq:bd2}), and (\ref{eq:bd3}), we have a homogeneous matrix equation
\begin{equation}
A \mathbf{c} = {0},
\label{eq:Det}
\end{equation}
where
\begin{equation*}
A = \left[
\begin{smallmatrix}
J_m\left(\sqrt{\frac{\lambda}{\beta^{+}}}R_O\right) & Y_{ m  }\left(\sqrt{\frac{\lambda}{\beta^{+}}}R_O\right) & \qquad 0 \\
J_{ m  }\left(\sqrt{\frac{\lambda}{\beta^{+}}}R_I\right) & Y_{ m  }\left(\sqrt{\frac{\lambda}{\beta^{+}}}R_I\right) & -J_{ m  }\left(\sqrt{\frac{\lambda}{\beta^{-}}}R_I\right) \\
\beta^{+}\frac{d}{d r}\left(J_{ m  }\left(\sqrt{\frac{\lambda}{\beta^{+}}}r\right)\right)|_{r=R_I} & \beta^{+}\frac{d}{d r}\left(Y_{ m  }\left(\sqrt{\frac{\lambda}{\beta^{+}}}r\right)\right)|_{r=R_I} & -\beta^{-}\frac{d}{d r}\left(J_{ m  }\left(\sqrt{\frac{\lambda}{\beta^{-}}}r\right)\right)|_{r=R_I}
\end{smallmatrix}
\right]
\end{equation*}
and $\mathbf c = [c_1^{+},  c_2^{+},  c_1^{-}]^{T}$. A nonzero solution of (\ref{eq:Det}) exists when the determinant of the matrix $A$ is zero.  For each index $m$, the eigenvalues $\lambda$ from (\ref{eq:model1}) coincide with the roots of the determinant of the matrix $A$, which can be easily calculated by any root-finding method such as the bisection method.
\end{Appendix}


\begin{thebibliography}{10}

\bibitem{Adams-Fournier}
{\sc R.~A.~Adams and J.~J.~F.~Fournier}, {\em Sobolev Spaces},
2nd ed., Elsevier, Amsterdam, 2003.

\bibitem{Antonietti-Buffa-Perugia}
{\sc P.~F.~Antonietti, A.~Buffa, and I.~Perugia}, {\em Discontinuous Galerkin approximation of the Laplace eigenproblem},
Comput. Methods Appl. Mech. Engrg. 195 (2006), pp. 3483--3503.

\bibitem{Arnold-Brezzi}
{\sc D.~N.~Arnold and F.~Brezzi}, {\em Mixed and nonconforming finite element methods: implementation, postprocessing and error estimates},
RAIRO Mod\'{e}l. Math. Anal. Num\'{e}r. 19 (1985), pp. 7--32.


\bibitem{Babuska-Osborn}
{\sc I.~Babu\v{s}ka and J.~E.~Osborn}, {\em Eigenvalue problems},
Handb. Numer. Anal. II, North-Holland, Amsterdam, 1991.

\bibitem{Badia-Codina}
{\sc S.~Badia and R.~Codina}, {\em A combined nodal continuous-discontinuous finite element formulation for the Maxwell problem},
Appl. Math. Comput. 218 (2011), pp. 4276--4294.

\bibitem{Boffi2007}
{\sc D.~Boffi}, {\em Approximation of eigenvalues in mixed form, discrete compactness property, and application to hp mixed finite elements},
Comput. Methods Appl. Mech. Engrg. 196 (2007), pp. 3672--3681.

\bibitem{Boffi-Brezi-Gastaldi}
{\sc D.~Boffi, F.~Brezzi, and L.~Gastaldi}, {\em On the problem of spurious eigenvalues in the approximation of linear elliptic problems in mixed form},
 Math. Comp. 69 (2000), pp. 121--140.

\bibitem{Bramble-King}
{\sc J.~H.~Bramble and J.~T.~King},{ \em A finite element method for interface problems in domains with smooth boundaries and interfaces},
Adv. Comput. Math. 6 (1996), pp. 109--138.

\bibitem{Brenner-Scott}
{\sc S.~C.~Brenner and L.~R.~Scott}, {\em The mathematical theory of finite element methods},
3rd ed., Texts Appl. Math. 15, Springer, New York, 2008.

\bibitem{Chang-Kwak}
{\sc K.~S.~Chang and D.~Y.~Kwak}, {\em Discontinuous bubble scheme for elliptic problems with jumps in the solution},
Comput. Methods Appl. Mech. Engrg. 200 (2011), pp. 494--508.

\bibitem{Chou-Kwak-Wee}
{\sc S.~H.~Chou, D.~Y.~Kwak, and K.~T.~Wee}, {\em Optimal convergence analysis of an immersed interface finite element method},
Adv. Comput. Math. 33 (2010), pp. 149--168.

\bibitem{Conway}
{\sc J.~B.~Conway}, {\em A course in functional analysis},
2nd ed., Springer-Verlag, Berlin, 1990.

\bibitem{Crouzeix-Raviart}
{\sc M.~Crouzeix and P.~A.~Raviart}, {\em Conforming and nonconforming finite element methods for solving the stationary Stokes equations I},
Rev. Fr. Autom. Inf. Rech. Oper. 7 (1973), pp. 33--75.

\bibitem{Dari-Duran-Padra}
{\sc E.~A.~Dari, R.~G.~Dur\'{a}n, and C.~Padra},
{\em A posteriori error estimates for non-conforming approximation of eigenvalue problems}, Appl. Numer. Math. 62 (2012), pp. 580--591.

\bibitem{Deak-Ahmed}
{\sc B.~Deka and T.~Ahmed}, {\em Convergence of finite element method for linear second-order wave equations with discontinuous coefficients},
Numer. Methods Partial Differential Equations 29 (2013), pp. 1522--1542.

\bibitem{Descloux-Nassif-Rappaz1978-1}
{\sc J.~Descloux, N.~Nassif, and J.~Rappaz}, {\em On spectral approximation. I. The problem of convergence},
RAIRO Anal. Num\'{e}r. 12 (1978), pp. 97--112.

\bibitem{Descloux-Nassif-Rappaz1978-2}
{\sc J.~Descloux, N.~Nassif, and J.~Rappaz}, {\em On spectral approximation. II. Error estimates for the Galerkin method},
RAIRO Anal. Num\'{e}r. 12 (1978), pp. 113--119.

\bibitem{Giani-Graham}
{\sc S.~Giani and I.~G.~Graham}, {\em A convergent adaptive method for elliptic eigenvalue problems},
SIAM J. Numer. Anal. 47 (2009), pp. 1067--1091.

\bibitem{Gong-Li-Li}
{\sc Y.~Gong, B.~Li, and Z.~Li}, {\em Immersed-interface finite-element methods for elliptic interface problems with nonhomogeneous jump conditions},
SIAM J. Numer. Anal. 46 (2008), pp. 472--495.

\bibitem{Hiptmair-Li-Zou}
{\sc R.~Hiptmair, J.~Li, and J.~Zou}, {\em Convergence analysis of finite element methods for $H(curl; \Omega)$-elliptic interface problems},
Numer. Math. 122 (2012), pp. 557--578.

\bibitem{Hou-Liu}
{\sc S.~Hou and X.-D.~Liu}, {\em A numerical method for solving variable coefficient elliptic equation with interfaces},
 J. Comput. Phys. 202 (2005), pp. 411--445.

\bibitem{Kato}
{\sc T.~Kato}, {\em Perturbation theory for linear operators},
Classics in Mathematics, Springer-Verlag, Berlin, 1995.

\bibitem{Kwak-W-C}
{\sc D.~Y.~Kwak, K.~T.~Wee, and K.~S.~Chang},
{\em An analysis of a broken $P_1$-nonconforming finite element method for interface problems}, SIAM J. Numer. Anal. 48 (2010), pp. 2117--2134.

\bibitem{Ladyzenskaja-Rivkind-Uralceva}
{\sc O.~A.~Lady\v{z}enskaja, V.~J.~Rivkind, and N.~N.~Ural'ceva}, {\em Solvability of diffraction problems in the classical sense},
 Trudy Mat. Inst. Steklov. 92 (1966), pp. 116--146.


\bibitem{Larson}
{\sc M.~G.~Larson}, {\em A posteriori and a priori error analysis for finite element approximations of self-adjoint elliptic eigenvalue problems},
 SIAM J. Numer. Anal. 38 (2000), pp. 608--625.

\bibitem{Lehoucq-Sorensen-Yang}
{\sc R.~B.~Lehoucq, D.~C.~Sorensen, and C.~Yang}, {\em ARPACK users' guide: solution of large-scale eigenvalue problems with implicitly restarted Arnoldi methods},
SIAM, Philadelphia, 1998.

\bibitem{LeVeque-Li}
{\sc R.~J.~Leveque and Z.~Li}, {\em The immersed interface method for elliptic equations with discontinuous coefficients and singular sources},
 SIAM J. Numer. Anal. 31 (1994), pp. 1019--1044.

\bibitem{Li-Lin-Lin-Rogers}
{\sc Z.~Li, T.~Lin, Y.~Lin, and R.~C.~Rogers}, {\em An immersed finite element space and its approximation capability},
 Numer. Methods Partial Differential Equations 20 (2004), pp. 338--367.

\bibitem{Li-Lin-Wu}
{\sc Z.~Li, T.~Lin, and X.~Wu}, {\em New Cartesian grid methods for interface problems using the finite element formulation},
 Numer. Math. 96 (2003), pp. 61--98.


\bibitem{Lin-Sheen-Zhang}
{\sc T.~Lin, D.~Sheen, and X.~Zhang}, {\em A locking-free immersed finite element method for planar elasticity interface problems},
J. Comput. Phys. 247 (2013), pp. 228--247.


\bibitem{Mercier-Osborn-Rappaz-Raviart}
{\sc B.~Mercier, J.~E.~Osborn, J.~Rappaz, and P.~A.~Raviart}, {\em Eigenvalue approximation by mixed and hybrid methods}, Math. Comp. 36 (1981), pp. 427--453.

\bibitem{Osborn}
{\sc J.~E.~Osborn}, {\em Spectral approximation for compact operators}, Math. Comput. 29 (1975), pp. 712--725.

\bibitem{Zhang-Leveque}
{\sc C.~Zhang and R.~J.~Leveque}, {\em The immersed interface method for acoustic wave equations with discontinuous coefficients},
Wave Motion 25 (1997), pp. 237--263.

\end{thebibliography}
\end{document}